\documentclass[11pt]{article}

\usepackage{etex}
\usepackage[margin={1.2in}]{geometry}
\usepackage[titletoc]{appendix}

\usepackage{rotating}

\usepackage{longtable}

\usepackage[sc]{mathpazo}
\usepackage{courier}

\usepackage[utf8]{inputenc}
\usepackage[T1]{fontenc}

\usepackage[english]{babel}
\usepackage{blindtext}

\usepackage{amsmath}

\usepackage{microtype}

\usepackage{amsmath}
\usepackage{amssymb}
\usepackage{amsthm}
\usepackage{color}
\usepackage{latexsym,extarrows}

\usepackage[all]{xy}
\usepackage{multirow}

\numberwithin{equation}{section}

\usepackage[stable,multiple]{footmisc}

\RequirePackage[font=small,format=plain,labelfont=bf,textfont=it]{caption}
\begin{document}
\def\e#1\e{\begin{equation}#1\end{equation}}
\def\ea#1\ea{\begin{align}#1\end{align}}
\def\eq#1{{\rm(\ref{#1})}}
\theoremstyle{plain}
\newtheorem{thm}{Theorem}[section]
\newtheorem{lem}[thm]{Lemma}
\newtheorem{prop}[thm]{Proposition}
\newtheorem{cor}[thm]{Corollary}
\theoremstyle{definition}
\newtheorem{dfn}[thm]{Definition}
\newtheorem{example}[thm]{Example}
\newtheorem{remark}[thm]{Remark}
\newtheorem{conjecture}{Conjecture}

\newtheorem*{mainthmA}{{\bf Theorem A}}
\newtheorem*{maincorC}{{\bf Corollary C}}
\newtheorem*{mainthmB}{{\bf Theorem B}}
\newtheorem*{mainthmD}{{\bf Theorem D}}



\newcommand{\bP}{\mathbb{P}}
\newcommand{\der}{\mathrm{d}}
\newcommand{\area}{\ensuremath \mathrm{Area}}
\newcommand{\fc}{\ensuremath F^{\mathrm{cyc}}}
\newcommand\cut{\operatorname{cut}}
\newcommand\glue{\operatorname{glue}}
\newcommand\form[1]{\langle #1\rangle}
\newcommand{\PP}{\mathbb{P}}
\newcommand{\PO}{\mathbb{P}^1_{a,b,c}}
\newcommand{\orbP}{\mathbb{P}^1_{a,b,c}}
\newcommand{\C}{\mathbb{C}}
\newcommand{\R}{\mathbb{R}}
\newcommand{\Q}{\mathbb{Q}}
\newcommand{\bH}{\mathbb{H}}

\newcommand{\BL}{\mathcal{L}}

\newcommand{\HH}{\mathbb{H}^2}
\newcommand{\DD}{\mathbb{D}^2}
\newcommand{\kk}{\boldsymbol{k}}
\newcommand{\Z}{\mathbb{Z}}
\newcommand{\nat}{\mathbb{N}}
\newcommand{\pd}{\partial}
\newcommand{\NN}{\mathbb{N}}
\newcommand{\AI}{A_\infty}
\newcommand{\LI}{L_\infty}
\newcommand{\CP}{\mathcal{P}}
\newcommand{\RP}{\mathbb{R}P}
\newcommand{\CE}{\mathcal{E}}
\newcommand{\CF}{\mathcal{F}}
\newcommand{\CO}{\mathcal{O}}
\newcommand{\cM}{\mathcal{M}}
\newcommand{\CA}{\mathcal{A}}
\newcommand{\CL}{\mathcal{L}}
\newcommand{\CH}{\mathcal{H}}
\newcommand{\CG}{\mathcal{G}}
\newcommand{\CJ}{\mathcal{J}}
\newcommand{\CS}{\mathcal{S}}
\newcommand{\OM}{\Omega^\bullet(M)}
\newcommand{\OBL}{\Omega^\bullet(L)}
\newcommand{\bx}{\boldsymbol{x}}
\newcommand{\by}{\boldsymbol{y}}
\newcommand{\bX}{\boldsymbol{X}}
\newcommand{\bi}{\mathbf{i}}
\newcommand{\E}{\epsilon}
\newcommand{\GG}{\Gamma}
\newcommand{\ds}{\displaystyle}
\newcommand{\WH}[1]{\widehat{#1}}
\newcommand{\WT}[1]{\widetilde{#1}}
\newcommand{\OL}[1]{\overline{#1}}
\newcommand{\CR}{\textit{crit}}
\newcommand\cay{\operatorname{Cay}}
\newcommand{\Jac}{\mathrm{Jac}}
\newcommand{\orb}{\mathrm{orb}}

\newcommand{\one}{\mathbb{1}}

\newcommand{\Span}{\mathrm{Span}}

\newcommand{\largewedge}{\mbox{\Large $\wedge$}}
\newcommand{\cL}{\mathcal{L}}

\newcommand{\even}{\mathrm{even}}
\newcommand{\odd}{\mathrm{odd}}

\makeatletter
\newcommand{\subjclass}[2][2010]{%
  \let\@oldtitle\@title%
  \gdef\@title{\@oldtitle\footnotetext{#1 \emph{Mathematics subject classification.} #2}}%
}
\newcommand{\keywords}[1]{%
  \let\@@oldtitle\@title%
  \gdef\@title{\@@oldtitle\footnotetext{\emph{Key words and phrases.} #1.}}%
}
\makeatother

\title{\bf Modularity of Open Gromov-Witten Potentials of Elliptic Orbifolds}
\author{ Siu-Cheong Lau and Jie Zhou}
\date{}

\maketitle

\begin{abstract}
We study the modularity of the genus zero open Gromov-Witten potentials and its generating matrix factorizations for elliptic orbifolds. These objects constructed by Lagrangian Floer theory are a priori well-defined only around the large volume limit. It follows from modularity that they can be analytically continued over the global K\"ahler moduli space.
\end{abstract}



\section{Introduction}

The mirror of an elliptic $\bP^{1}$ orbifold $\bP^1_{a,b,c}$ is a Landau-Ginzburg mirror: it is determined by a polynomial
\begin{equation}\label{eqmirrorcurvefamily}
 W^{\textrm{mir}} = x^a + y^b + z^c + \sigma xyz \,,
\end{equation}
where $\sigma$ is a complex parameter.  Mirror symmetry asserts that symplectic geometry of $\bP^1_{a,b,c}$ is reflected from the complex geometry of $W^{\textrm{mir}}$, and vice versa.  While the orbifold $\bP^1_{a,b,c}$ is only of dimension one, its Gromov-Witten theory is very interesting and receives a lot of attention in the context of mirror symmetry and integrable systems, see for instance \cite{milanov-tseng08, Ta, Rossi, Milanov:2011, Satake:2011, KrSh, Milanov:2012, ET, Li:2013, Shen:2014}.

The paper \cite{CHL} proposed a systematic construction of Landau-Ginzburg mirror and a homological mirror functor using Lagrangian Floer theory.  For an elliptic $\bP^1$ orbifold $\bP^1_{a,b,c}$, where $\frac{1}{a} + \frac{1}{b} + \frac{1}{c} = 1$, the construction produces a polynomial $W_q(x,y,z)$ whose coefficients are convergent series in the K\"ahler parameter $q$ of $\bP^1_{a,b,c}$. The polynomial $W_q$ can be rearranged to the form of $W^{\textrm{mir}}$ by an explicit change of coordinates in $(x,y,z)$.  It is called to be the open Gromov-Witten potential because it is obtained by counting holomorphic polygons bounded by a fixed Lagrangian, which is the Lagrangian immersion constructed by Seidel \cite{Seidel:g=2}.

The open Gromov-Witten potential $W_q(x,y,z)$ is a priori defined only around the point $q=0$, the so-called large volume limit of the K\"ahler moduli space.  In this paper, we show that indeed it can be extended to certain global moduli space:

\begin{thm} \label{thm:main}
Let $W_q(x,y,z)$ be the open Gromov-Witten potential of an elliptic $\bP^1$ orbifold $\bP^1_{a,b,c}$ where $(a,b,c) = (3,3,3)$ or $(2,4,4)$. The coefficients of $W_q(x,y,z)$, which are functions in $q$, are modular forms of certain weight $k$ for the
modular group $\Gamma = \Gamma(3)$ or $\Gamma(4)$ respectively. Hence the potential extends to be a section of
the line bundle $\mathcal{K}^{k\over 2}$ over the product $\C^3 \times \left(\Gamma\backslash \mathcal{H}^{*}\right)$, where $\mathcal{K}$ is the pull back of the canonical line bundle of the modular curve $\Gamma\backslash \mathcal{H}^{*}$.
\end{thm}

The proof is arithmetic in nature. We explicitly express the open Gromov-Witten potential in terms of the Dedekind $\eta$-function and Eisenstein series, and use known expressions for modular forms with respect to the groups $\Gamma = \Gamma(3)$ and $\Gamma(4)$.  We expect the same statement holds for the case $(a,b,c) = (2,3,6)$, see Section \ref{sec:236} for more details.

\begin{remark}
The theorem also holds for the elliptic orbifold $\bP^1_{2,2,2,2}$, namely the coefficients of the open Gromov-Witten potential of $\bP^1_{2,2,2,2}$ are modular forms for the modular group $\Gamma(2)$.  See Section \ref{sec:2222}.  In this case $W$ is defined on the resolved conifold $\CO_{\bP^1}(-1)\oplus\CO_{\bP^1}(-1)$ rather than $\C^3$, and its critical locus is the zero section $\bP^1 \subset \CO_{\bP^1}(-1)\oplus\CO_{\bP^1}(-1)$ rather than an isolated point.  Thus we separate this case from the above theorem.
\end{remark}

For an elliptic $\bP^1$ orbifold $\bP^1_{a,b,c}$, the mirror functor produces a particular matrix factorization $M$ of $W_q$, which is an odd endomorphism $\delta$ on $\largewedge^* \C^3$ satisfying $\delta^2 = W_q \cdot \mathrm{Id}$. This matrix factorization has the important property that it split generates the derived category of matrix factorizations, and it is mirror to the Seidel Lagrangian. Using similar arithmetic techniques, we can express $M$ in terms of modular forms.

\begin{thm}
Let $M$ be the matrix factorization of the open Gromov-Witten potential $W_q(x,y,z)$ which is mirror to the Seidel Lagrangian in $\bP^1_{a,b,c}$, where $(a,b,c) = (3,3,3)$ or $(2,4,4)$. The matrix entries of $M$ are polynomials in $x,y,z$ whose coefficients are modular forms of weight $k$ for the
modular group $\Gamma = \Gamma(3)$ or $\Gamma(4)$ respectively.
\end{thm}

Why modularity is expected can be explained as follows.  The Seidel Lagrangian in the elliptic orbifold $\bP^{1}_{a,b,c}=E/\Z_r$, where $r=3,4,6$ for the $(a,b,c)=(3,3,3),(2,4,4),(2,3,6)$ case respectively, can be lifted to $r$ copies of Lagrangians in the elliptic curve $E$. Thus the moduli space around the large volume limit under consideration on the symplectic side is the moduli space of K\"ahler structure of $E$ \emph{together with a particular choice of $r$ Lagrangians}. The mirror is the family of elliptic curves decorated with structures of $r$-torsion points, whose moduli space turns out to be the modular curve $\Gamma\backslash \mathcal{H}^{*}$. Mirror symmetry asserts that the A-side moduli is globally isomorphic to the B-side moduli.  Thus $\Gamma\backslash \mathcal{H}^{*}$ should also be the global K\"ahler moduli. Our results confirm that the open Gromov-Witten potential, which is originally just defined around the large volume limit, naturally extend to this global K\"ahler moduli space.

\begin{remark}
Modularity of closed Gromov-Witten potentials for elliptic curves and elliptic orbifolds is derived in a series of works including \cite{Dijkgraaf:1995, Kaneko:1995, Eskin:2001, Okunkov:2002, Milanov:2011, Satake:2011,  Li:2011mi, Shen:2014}. 
For discussions on modularity of some higher dimensional Calabi-Yau varieties, interested readers are referred to \cite{Bershadsky:1993ta, Bershadsky:1993cx, Antoniadis:1995zn, Kachru:1995wm, Marino:1998pg, Klemm:2004km, Klemm:2005pd, Aganagic:2006wq, Alim:2012ss, Klemm:2012sx, Alim:2013eja, Pandharipande:2014} and references therein for details.
\end{remark}

\subsection*{Structure of the paper}
In Section \ref{sectionmodularforms}, we review some basic materials on modular forms and elliptic curve families defined over some modular curves. In Section \ref{sectionGWpotentials}, we recall the construction of the Seidel Lagrangian and 
prove the modularity for the potentials $W$. In Section \ref{sectionmatrixfactorizations}, we prove the modularity for the matrix factorizations $M$.
We discuss why modularity is expected from the perspective of mirror symmetry and give one further example in Section \ref{sectionexplanation}.

\subsection*{Acknowledgment}
We are grateful to Shing-Tung Yau for constant support and encouragement.
We thank Kathrin Bringmann and Larry Rolen for email correspondences and helpful discussions on mock modular forms.  
The second author thanks Murad Alim, Emanuel Scheidegger, and Shing-Tung Yau for fruitful collaborations on related projects.
He also thanks Teng Fei, Todor Milanov, Yongbin Ruan, Yefeng Shen for very useful discussions on various aspects of elliptic orbifolds and modular forms.
We thank the referees for helpful comments and pointing out the importance of the study of the elliptic orbifold $\bP^1_{2,2,2,2}$\,.

Part of the work is done while the second author was a graduate student at the mathematics department at Harvard.  We would like to thank the department for providing an excellent research atmosphere. J. Z. is supported by the Perimeter Institute for Theoretical Physics. Research at Perimeter Institute is  supported  by the Government  of Canada  through  Industry Canada and  by the Province of Ontario through the Ministry of Economic Development and Innovation.


\section{Preliminaries on modular forms} \label{sectionmodularforms}

In this section we give a quick review on some background material about modular forms and modular curves. They are essential to our study because global K\"ahler moduli space of elliptic orbifolds will be identified as modular curves by using mirror symmetry. The open Gromov-Witten potentials and matrix factorizations will be written in terms of modular forms, which are global sections of the corresponding line bundles over modular curves.
The material presented here is largely taken from a joint work \cite{Alim:2013eja} of the second author.

Throughout this paper, we fix $q=\exp 2\pi i\tau$, with $\tau$ is the coordinate on the upper-half plane $\mathcal{H}$. The quantity $-2\pi i
\tau$ can be regarded as parametrizing the (complexified) symplectic
area of an elliptic orbifold (and so $q$ defines a local coordinate near the large volume limit $q=0$ on
the complexified K\"ahler moduli space of the elliptic orbifold).

\subsection{Modular groups and modular forms}

The generators and relations for the group
$\mathrm{SL}(2,\mathbb{Z})$ are given by the following:
\begin{equation}
  \label{Sl2Z}
T =
  \begin{pmatrix}
    1 & 1\\ 0 & 1
  \end{pmatrix}\,,\quad
  S=
  \begin{pmatrix}
    0& -1 \\ 1 & 0\\
  \end{pmatrix}\,,\quad S^{2}=-I\,,\quad (ST)^{3}=-I\,.
\end{equation}
We will consider in this paper the following congruence subgroups called Hecke subgroups of 
$\Gamma(1)=\mathrm{PSL}(2,\mathbb{Z})=\mathrm{SL}(2,\mathbb{Z})/\{\pm I\}$
\begin{equation}
\Gamma_{0}(N)=\left\{ \left.
\begin{pmatrix}
a & b  \\
c & d
\end{pmatrix}\in \Gamma(1)
\right\vert\, c\equiv 0\,~ \textrm{mod} \,~ N\right\}< \Gamma(1)\,.
\end{equation}
Some other groups that we are interested in are the principal
congruence subgroups
\begin{equation}
\Gamma(N)=\left\{ \left.
\begin{pmatrix}
a & b  \\
c & d
\end{pmatrix}\in \Gamma(1)
\right\vert\,
\begin{pmatrix}
a & b  \\
c & d
\end{pmatrix}
\equiv
\begin{pmatrix}
1 & 0 \\
0 & 1
\end{pmatrix}
\,~ \textrm{mod} \,~ N\right\}< \Gamma(1)\,.
\end{equation}
One has $\Gamma(N) <  \Gamma_0(N)<\Gamma(1)=\mathrm{PSL}(2,\mathbb{Z})$.

A modular form of weight $k$ for the congruence subgroup $\Gamma$ of $\mathrm{PSL}(2,\mathbb{Z})$ is a
function $f : \mathcal{H}\rightarrow \mathbb{C}$
satisfying the following conditions:
\begin{itemize}
\item $\,f(\gamma \tau)=j_{\gamma}(\tau)^{k}f(\tau), \quad \forall \gamma\in \Gamma\,$, where $j$ is called the $j$-automorphy factor and is defined by
$$j: \Gamma \times \mathcal{H}\rightarrow
\mathbb{C},\quad \left(\gamma=\left(
\begin{array}{cc}
a & b  \\
c & d
\end{array}
\right),\tau\right)\mapsto j_{\gamma}(\tau):=(c\tau+d)\,.$$
\item $\,f$ is holomorphic on $\mathcal{H}$.
\item $\,f$ is ``holomorphic at the cusps" in the sense that the function
\begin{equation}\label{slash}
 \tau\mapsto j_{\gamma}(\tau)^{-k}
f(\gamma\tau)
\end{equation}
is holomorphic at $\tau=i\infty$ for any $\gamma\in \Gamma(1)$.
\end{itemize}
The second and third conditions in the above can be equivalently
described as saying that $f$ is holomorphic on the modular curve
$X_{\Gamma}=\Gamma\backslash \mathcal{H}^{*}$, where
$\mathcal{H}^{*}=\mathcal{H}\cup \mathbb{P}^{1}(\mathbb{Q})$, i.e.,
$\mathcal{H}\cup \mathbb{Q}\cup \{i\infty\}$. The first condition means that $f$ can be formulated as a holomorphic section of a line bundle over $X_{\Gamma}$.

We will need to be able to take roots of modular forms. For this purpose we introduce the notion of multiplier system. A multiplier system of weight $k$ for $\Gamma$
is a function $v:\Gamma\rightarrow \mathbb{C}$ such that
$|v(\gamma)|=1$ and $v(\gamma_{1}\gamma_{2})=w(\gamma_{1},\gamma_{2})v(\gamma_{1})v(\gamma_{2})$ for some $w(\gamma_{1},\gamma_{2})$.
We then define modular forms of weight $k$ with the multiplier system $v$ by replacing the $j$-automorphy factor in
\eqref{slash} by the new automorphy factor $v(\gamma)j_{\gamma}(\tau)$,
see for example
\cite{Rankin:1977ab} for details. 
The simplest case is when $v$ depends only on the entry $d$ of $\gamma$.
In the following we will be mostly dealing with the case where $v$
is given by a Dirichlet character $\chi$.
The space of modular forms with the
multiplier system $\chi$ for $\Gamma$ forms a graded differential
ring and is denoted by $M_{*}(\Gamma, \chi)$. Similarly we have the ring of
even weight modular forms  denoted by $M_{\mathrm{even}}(\Gamma,\chi)$. When $\chi$ is trivial,
we shall often omit it and simply write $M_{*}(\Gamma)$.

\begin{example}
Taking the group $\Gamma$ to be the full modular group
$\Gamma(1)=\mathrm{PSL}(2,\mathbb{Z})$. Then
$M_{*}(\Gamma(1))=\mathbb{C}[E_{4},E_{6}]$, where $E_{4},E_{6}$ are
the familiar Eisenstein series defined by
\begin{align*}
E_{4}(\tau)&=1+240\sum_{d=1}^{\infty}\sigma_{3}(d)q^{d},\quad q=e^{2
\pi
i\tau},\quad \sigma_{3}(d)=\sum_{k:\,k|d}k^{3}\,,\\
E_{6}(\tau)&=1-504\sum_{d=1}^{\infty}\sigma_{3}(d)q^{d},\quad q=e^{2
\pi i\tau},\quad \sigma_{5}(d)=\sum_{k:\,k|d}k^{5}\,.
\end{align*}
The Eisenstein series
$E_{2}(\tau)=1-24\sum_{d=1}^{\infty}\sigma_{1}(d)q^{d}$ is not a
modular form, but a so-called quasi-modular form \cite{Kaneko:1995}
for $\Gamma(1)$, since it transforms according to
\begin{equation*}
E_{2}({a\tau+b\over c\tau+d})=(c\tau+d)^{2}E_{2}(\tau)+{12\over 2\pi
i}c(c\tau+d),\quad \forall ~\tau\in \mathcal{H},\quad \forall
\begin{pmatrix}
a & b  \\
c & d
\end{pmatrix} \in \Gamma(1)\,.
\end{equation*}
\end{example}

\subsection{Ring of modular forms}\label{sectionringofmodularforms}

Now we consider modular forms (with possibly non-trivial multiplier
systems) for the Hecke subgroups $\Gamma_{0}(N)$ with $N=2,3,4$ and
the subgroup $\Gamma_{0}(1^{*})$ which is the unique index two
normal subgroup of $\Gamma(1)=\mathrm{PSL}(2,\mathbb{Z})$. 
All of them are of
genus zero in the sense that the corresponding modular curves\footnote{The $N=1^{*}$ case is anomalous, more details are given in Section \ref{sectiongeometricmoduli}. For further discussion, see \cite{Maier:2009}.} $X_{0}(N):=\Gamma_{0}(N)\backslash \mathcal{H}^{*}$ are genus zero
Riemann surfaces. Each of the corresponding modular curves
$X_{\Gamma}$ has three singular points: two (equivalence classes) of
cusps\footnote{Here we use the notation $[\tau]$ to denote the $\Gamma$-equivalence class of $\tau\in \mathcal{H}^{*}$.} $[i\infty],[0]$, and the third one is a cusp or an
elliptic point, depending on the modular group. It is a quadratic
elliptic point $[\tau]=[i]$ for $N=2$, cubic elliptic point $[\tau]=[\exp
2\pi i/3]$ for $N=3$ and $N=1^{*}$, and a cusp $[\tau]=[1/2]$ for $N=4$.
For a review of these facts, see for instance \cite{Rankin:1977ab}.

We can choose a particular Hauptmodul (i.e., a generator for the
rational function field of the genus zero modular curve)
$\alpha(\tau)$ for the corresponding modular group such that the two
cusps are given by $\alpha=0,1$ respectively, and the third one is
$\alpha= \infty$. It is given by $\alpha(\tau)=C^{r}(\tau)/
A^{r}(\tau)$, where $r=6,4,3,2$ for the cases $N=1^{*},2,3,4$
respectively.  The functions\footnote{Throughout this work, when we take factional powers of modular forms and modular functions, we always take the principal branch of the logarithm.} $A(\tau),C(\tau)=\alpha(\tau)^{1\over
r}A(\tau), B(\tau)=(1-\alpha(\tau))^{1\over r}A(\tau)$ are given in
Table \ref{tableetaexpansions} below.
\begin{table}[h]
  \centering
  \caption[$\eta$-expansions of $A,B,C$ for $\Gamma_{0}(N), N=1^{*},2,3,4$]{$\eta$-expansions of $A,B,C$ for $\Gamma_{0}(N), N=1^{*},2,3,4$}
  \label{tableetaexpansions}
  \renewcommand{\arraystretch}{1.2} 
 \begin{tabular}{c|ccc}
$N$&$A$&$B$&$C$\\
\hline
$1^{*}$&$E_{4}(\tau)^{1\over 4}$&$({E_{4}(\tau)^{3\over 2}+E_{6}(\tau)\over 2})^{1\over 6}$&$({E_{4}(\tau)^{3\over 2}-E_{6}(\tau)\over 2})^{1\over 6}$\\
$2$&${(2^{6}\eta(2\tau)^{24}+\eta(\tau)^{24} )^{1\over 4} \over \eta(\tau)^2\eta(2\tau)^2}$&${\eta(\tau)^{4}\over \eta(2\tau)^{2}}$&$2^{3 \over 2}{\eta(2\tau)^4\over \eta(\tau)^2}$ \\
$3$&${(3^{3}\eta(3\tau)^{12}+\eta(\tau)^{12} )^{1\over 3} \over \eta(\tau)\eta(3\tau)}$&${\eta(\tau)^{3}\over \eta(3\tau)}$&$3{\eta(3\tau)^3\over \eta(\tau)}$ \\
$4$&${(2^{4}\eta(4\tau)^{8}+\eta(\tau)^{8} )^{1\over 2} \over
\eta(2\tau)^2}=
{\eta(2\tau)^{10}\over\eta(\tau)^{4}\eta(4\tau)^{4}}$&${\eta(\tau)^{4}\over
\eta(2\tau)^2}$&$2^2{\eta(4\tau)^4\over \eta(2\tau)^2}$
\end{tabular}
\end{table}
See \cite{Borwein:1991,
Berndt:1995} and also \cite{Maier:2009, Maier:2011} for a review on the modular forms $A,B,C$.
Throughout this paper we shall write $A_N, B_N, C_N$ for the $\Gamma = \Gamma_0(N)$
case for these quantities when potential confusion might arise. 

The explicit expressions for these quantities in terms of $\theta$-functions and $q$-series can be found in a lot of literature.
By using the $\theta$-expansions therein for these generators, one can easily see that
\begin{equation}
A_{2}^{2}=A_{4}^{2}+C_{4}^{2}\,,\quad
C_{2}^{2}=2A_{4}C_{4}\,.
\end{equation} 
The
following results are classical:
\begin{align*}
M_{\mathrm{even}}(\Gamma_{0}(2))&=\mathbb{C}[A_{2}^{2},B_{2}^{4}]\,,\\
M_{*}(\Gamma_{0}(3),\chi_{-3})&=\mathbb{C}[A_{3},B_{3}^{3}]\,,\\
M_{*}(\Gamma_{0}(4),\chi_{-4})&=\mathbb{C}[A_{4},B_{4}^{2}]\,.
\end{align*}
Here $\chi_{-3}(d)=\left(\frac{-3}{d}\right)$ is the Legendre symbol
and it gives the non-trivial Dirichlet character for the modular
forms. Similarly, $\chi_{-4}(d)=\left(\frac{-4}{d}\right)$. From
these we can derive the following results:
\begin{align}
M_{*}(\Gamma(3))&=\mathbb{C}[A_{3},C_{3}]\,,\\
M_{*}(\Gamma(4))&=\mathbb{C}[A_{4},C_{4},C_{2}]/\langle C_{2}^{2}-2A_{4}C_{4} \rangle\,.
\end{align}
For the modular group $\Gamma(2)$, the ring of modular forms is isomorphic to that for $\Gamma_{0}(4)$ by using the $2$-isogeny which gives an isomorphism between the modular groups.
See for instance \cite{Bannai:2001, Sebbar:2002, Maier:2011} and references therein for
details of all these results.

\subsection{Geometric moduli in terms of modular forms}\label{sectiongeometricmoduli}

In this section, we shall discuss some basic facts about the geometry
and arithmetic of the elliptic curve families of $E_{n}$ type\footnote{The names come from the fact that the total spaces of the elliptic curve families correspond to the $E_{n}$ del Pezzo surfaces, see for instance \cite{Klemm:2012sx} for further explanation.}, $n=5,6,7,8$. 
They are
closely related to\footnote{In fact, for the $n=6,7,8$ cases these are, up to reparametrization, the simple elliptic singularities \cite{Saito:1974} $E_{6}^{(1,1)},E_{7}^{(1,1)}, E_{8}^{(1,1)}$ and are mirror to the elliptic orbifolds, see \cite{Milanov:2011}.} the elliptic orbifolds which are the main focus
of this work, as we shall see in the sequel.

The equations for the elliptic curve families are given by
\begin{eqnarray}
n=5&:&  \mathbb{P}^{3}[1,1,1,1][2,2]:x_{1}^{2}+x_{3}^{2}-z^{-{1\over 2r}}x_{2}x_{4}=0\nonumber\,,\\
&&\hspace{30mm} x_{2}^{2}+x_{4}^{2}-z^{-{1\over 2r}}x_{1}x_{3}=0\nonumber\,,\\
n=6&:&  \mathbb{P}^{2}[1,1,1][3]: x_{1}^{3}+x_{2}^{3}+x_{3}^{3}-z^{-{1\over r}}x_{1}x_{2}x_{3}=0\nonumber\,,\\
n=7&:&  \mathbb{P}^{2}[1,1,2][4]: x_{1}^{4}+x_{2}^{4}+x_{3}^{2}-z^{-{1\over r}}x_{1}x_{2}x_{3}=0\nonumber\,,\\
n=8&:&  \mathbb{P}^{2}[1,2,3][6]:
x_{1}^{6}+x_{2}^{3}+x_{3}^{2}-z^{-{1\over r}}x_{1}x_{2}x_{3}=0\,,
\end{eqnarray}
where the numbers $r$ are given by $2,3,4,6$ for $n=5,6,7,8$,
respectively.\\

The $j$-invariants for these elliptic curve families are summarized
here, see \cite{Lian:1994zv,Lerche:1996ni, Klemm:1996,
Chiang:1999tz} for more details.  
\begin{equation}
E_{5}: \, \left\{
\begin{array}{c l}
x_{1}^{2}+x_{3}^{2}-z^{-{1\over 4}} x_{2}x_{4}&=0\\
x_{2}^{2}+x_{4}^{2}-z^{-{1\over 4}} x_{1}x_{3}&=0
\end{array}\right.
\quad j(z)={(1+224z+256z^{2})^{3}\over z(1-16z)^{4}}\,.
\end{equation}

The base of this family of elliptic curves is the modular curve
$X_{0}(4)$. It has three singular points: two cusp classes
$[i\infty],[0]$ corresponding to $z=0,1/16$ respectively; and the
cusp class $[1/2]$ corresponding to $z=\infty$.
\begin{equation}
E_{6}: \,x_{1}^{3}+x_{2}^{3}+x_{3}^{3}-z^{-{1\over 3}}
x_{1}x_{2}x_{3}=0 \,,\quad j(z)={(1+216z)^{3}\over
z(1-27z)^{3}}\,.
\end{equation}

The base of this family of elliptic curves is the modular curve
$X_{0}(3)$. It has three singular points: two cusp classes
$[i\infty],[0]$ corresponding to $z=0,1/27$ respectively; and the
cubic elliptic point $[ST^{-1}(\rho)]$ corresponding to $z=\infty$,
where $\rho=\exp (2\pi i/3)$.

\begin{equation}
E_{7}:\,x_{1}^{4}+x_{2}^{4}+x_{3}^{2}-z^{-{1\over 4}}
x_{1}x_{2}x_{3}=0 \,,\quad j(z)={(1+192z)^{3}\over
z(1-64z)^{2}}\,.
\end{equation}

The base of this family of elliptic curves is the modular curve
$X_{0}(2)$. It has three singular points: two cusp classes
$[i\infty],[0]$ corresponding to $z=0,1/64$ respectively; and the
quadratic elliptic point $[(1+i)/2]=[ST^{-1}(i)]$ corresponding to
$z=\infty$.

\begin{equation}
E_{8}:\,x_{1}^{6}+x_{2}^{3}+x_{3}^{2}-z^{-{1\over 6}}
x_{1}x_{2}x_{3}=0\,,\quad j(z)={1\over z(1-432z)}\,.
\end{equation}

The base of this family of elliptic curves is the curve
$X_{0}(1^{*})=\Gamma_{0}(1^{*})\backslash \mathcal{H}^{*}$, where
$\Gamma_{0}(1^{*})$ is the unique index $2$ normal subgroup of
$\Gamma(1)=\mathrm{PSL}(2,\mathbb{Z})$. It has three singular points: two
cusp classes $[i\infty],[0]$ corresponding to $z=0,1/432$
respectively; and the cubic elliptic point $[\rho]$ corresponding to
$z=\infty$.\\

The Hauptmodul for the corresponding modular group given in the
previous section is related to the parameter $z$ by
$\alpha=\kappa_{N} z$, where $\kappa_{N}$ is given $432,64,27,16$
for $n=8,7,6,5$ (i.e., $N=1^{*},2,3,4$), respectively.
For reference, we now summarize the related quantities in Table \ref{tablearithmeticnumbers} below.
\begin{table}[h]
  \centering
  \caption[Arithmetic numbers]{Arithmetic numbers}
  \label{tablearithmeticnumbers}
  \renewcommand{\arraystretch}{1.2} 
 \begin{tabular}{c|cccc}
$n$&5&6&7&8\\
\hline
$N$&4&3&2&$1^{*}$\\
$r$&2&3&4&6\\
$\kappa_{N}$&16&27&64&432\\
\end{tabular}
\end{table}
Here the number $r$ is given by $r=12/\nu$, with $\nu$ being the index of the subgroup in the full modular group $\mathrm{PSL}(2,\mathbb{Z})$.

\begin{remark}
The Picard-Fuchs operators of the above elliptic curves of $E_{n}$
type have the form
\begin{equation}\label{PFforonefold}
\mathcal{L}_{\mathrm{Picard-Fuchs}}=\theta^{2}-\alpha
(\theta+{1\over r})(\theta+1-{1\over r}),\quad
\theta=\alpha{\partial\over \partial \alpha}\,.
\end{equation}
Denote $A(\alpha)=\,_{2}F_{1}({1\over r},1-{1\over r},1;\alpha)$ to be
the regular period at $\alpha=0$ of the elliptic curve family. Then
the modular form $A(\tau)$ given in the previous section is actually
given by $A(\alpha(\tau))$. One also has $\tau(\alpha)={i\over
\sqrt{N}}A(1-\alpha)/A(\alpha)$. Therefore, the triple $A(\tau),
B(\tau),C(\tau)$ introduced earlier can be reconstructed from the
periods, see \cite{Borwein:1991, Berndt:1995, Maier:2009}. This fact was used in
\cite{Alim:2013eja, Zhou:2014thesis} in studying modularity in Gromov-Witten theory and mirror symmetry for some
non-compact Calabi-Yau threefolds.
\end{remark}

In Section \ref{sectionGWpotentials} and Section \ref{sectionmatrixfactorizations}
below, we will be mainly working with the A-model of the elliptic orbifolds, that is, studying the dependence of the generating functions of genus zero open Gromov-Witten invariants on the complexified K\"ahler structure.
In Section \ref{sectionexplanation}, we will comment on how mirror symmetry maps the symplectic geometry data of elliptic orbifolds to the complex geometry data of the elliptic curve families described in this section. This would then give an explanation of why modularity is expected.


\section{Open Gromov-Witten potentials of elliptic orbifolds}\label{sectionGWpotentials}

In this section, we study modularity of open Gromov-Witten potentials of elliptic orbifolds.  First let us have a quick glance on the construction of open Gromov-Witten potentials in \cite{CHL,CHKL} using immersed Lagrangian Floer theory.

Given a K\"ahler orbifold $X$, we fix a Lagrangian immersion $\BL$, which is assumed to be oriented and (relatively) spin, and not passing through the orbifold points of $X$.  Moreover we assumed that it has transverse self-intersections for simplicity.  Let $\iota: \tilde{\BL} \to X$ denote the normalization of $\BL$.  We assume that $\tilde{\BL}$ is connected.

We use the deformations and obstructions of $\BL$ to construct a Landau-Ginzburg model $(V,W)$.  It is called to be the generalized SYZ construction: it uses deformations of an immersed Lagrangian to construct the mirror, while SYZ uses a Lagrangian torus fibration for the same purpose.
The detailed deformation theory for Lagrangian immersion, which is captured by an $A_\infty$ algebra $\left(H,\{m_k\}_{k=0}^\infty\right)$, was developed in \cite{AJ}.  Here we only sketch the needed ingredients.

Each transverse self-intersection point $a$ corresponds to two immersed generators $X^0_a,X^1_a$ of the Floer complex of $\BL$.  Intuitively they represent the two ways of smoothings of the self-intersection point.  For a formal deformation 
\begin{equation}
b = \sum_a (x^0_a X^0_a + x^1_a X^1_a) \in H = \bigoplus_a \Span_\C \{X^0_a,X^1_a\}\,,
\end{equation}
where the sum is over all self-intersection points $a$,
we have the deformed $m_0$-term
\begin{equation}
m_0^b = \sum_{k=0}^\infty m_k(b,\ldots,b) = \sum_{k=0}^\infty \sum_{\substack{
(a_1,\ldots,a_k)\\(s_1,\ldots,s_k)}}  m_k(X^{s_1}_{a_1}, \ldots, X^{s_k}_{a_k}) x^{s_1}_{a_1} \ldots x^{s_k}_{a_k} \,,
\end{equation}
which is a singular chain in the fiber product $\tilde{\BL} \times_\iota \tilde{\BL}$.  Roughly speaking it is a sum of the boundary evaluation images of holomorphic polygons bounded by $\BL$ (weighted by their symplectic areas) with corners at the immersed generators.

Then we choose a subspace $V$ of $H$ whose elements $b \in V$ have odd degrees and satisfy the so-called weak Maurer-Cartan equation \cite{FOOO}
\begin{equation}
 m_0^b = W(b) \one_{\tilde{\mathcal{L}}}\,, 
\end{equation}
where $\one_{\tilde{\mathcal{L}}}$ denotes the fundamental class of $\tilde{\mathcal{L}}$.  Such deformations $b$ are called to be weakly unobstructed.  This defines a function $W$ on $V$, and we call it to be the open Gromov-Witten potential of $\mathcal{L}$ because coefficients of $W$ are obtained by counting pseudoholomorphic polygons bounded by the immersed Lagrangian $\mathcal{L}$.

To construct the open Gromov-Witten potential (or so-called Landau-Ginzburg mirror) of elliptic $\bP^1$ orbifolds, we take $\mathcal{L}$ to be the Lagrangian immersion constructed by Seidel \cite{Seidel:g=2}. It has three self-intersection points as depicted in Figure \ref{Seidel_Lag}. We take the formal deformations $b = xX + yY + zZ$, where $X,Y,Z$ are immersed generators of odd degrees as shown in the figure.  By \cite[Lemma 7.5]{CHL}, these deformations are weakly unobstructed.  Thus we obtain an open Gromov-Witten potential $W(x,y,z)$.

\begin{figure}[htp]
\begin{center}
\includegraphics[scale=0.35]{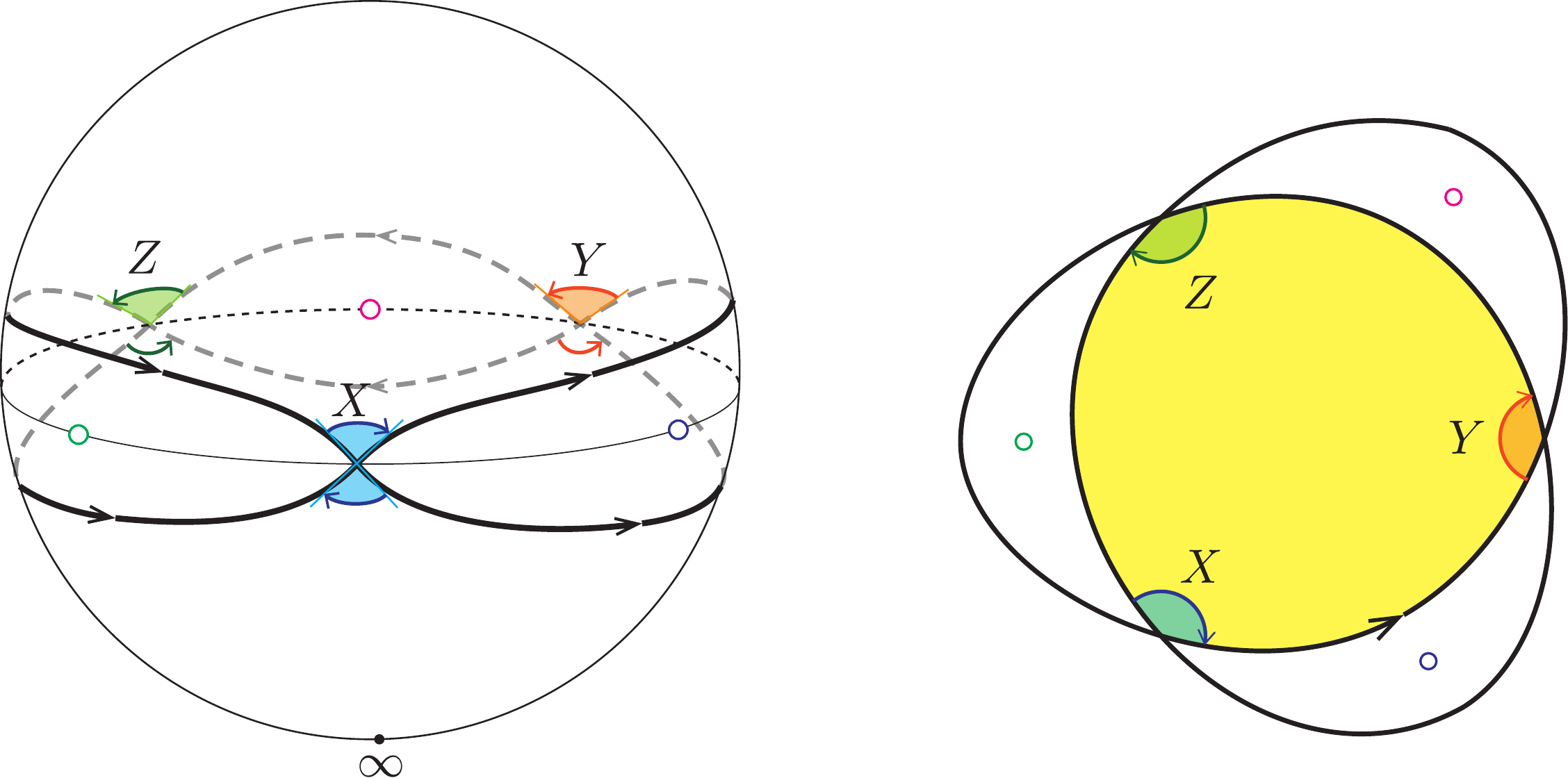}
\caption{The Seidel Lagrangian.  The two pictures above show the same Lagrangian immersion from different viewpoints.  The three dots on the equator represent orbifold points.  The shaded triangle on the right contributes to the term $-q_d xyz$ of the open Gromov-Witten potential, where $q_d = \exp (-A)$, $A$ is the symplectic area of the shaded triangle.}\label{Seidel_Lag}
\end{center}
\end{figure}

Note that the potential $W(x,y,z)$ depends on the K\"ahler parameter of the elliptic $\bP^1$ orbifold, which parametrizes the sizes of the holomorphic polygons.  Thus $W$ can be identified as a map from the K\"ahler moduli of the $\bP^1$ orbifold to the complex moduli of holomorphic functions.  We call this to be the generalized SYZ map because it arises from the generalized SYZ construction described above.

The explicit expression of $W$ and the generalized SYZ map were computed in \cite[Section 6.1]{CHL} for the elliptic orbifold $\bP^1_{3,3,3}$ and in \cite[Section 9 and 10]{CHKL} for the elliptic orbifolds $\bP^1_{2,4,4}$ and $\bP^1_{2,3,6}$. In the rest of this section we shall study modularity of the coefficients of the open Gromov-Witten potential.

\begin{remark}
There is another elliptic orbifold curve which is not listed above, namely $\bP^{1}_{2,2,2,2}$ which is the $\Z_2$-quotient of some elliptic curve.  A similar construction scheme for its open Gromov-Witten potential can be carried out, which involves more than one Lagrangian immersions.  The details about the construction of the open Gromov-Witten potential and the mirror functor will be given in a forthcoming work \cite{CHLnc}. In this paper, we will state the result of the open Gromov-Witten potential and discuss its modularity.
\end{remark}

\subsection{$(3,3,3)$ case}

\begin{thm}\cite{CHKL} \label{thm:W333}
The open Gromov-Witten potential for $\bP^1_{3,3,3}$ is
\begin{equation}
W = \phi(q_{\mathrm{d}}) (x^3 + y^3 + z^3) - \psi(q_{\mathrm{d}}) xyz \,,
\end{equation}
where
\begin{equation} 
\phi(q_{\mathrm{d}}) = \sum_{k=0}^{\infty} (-1)^{3k+1} (2k+1) q_{\mathrm{d}}^{3(12 k^2 + 12 k + 3)}\,,
\end{equation}
and
\begin{equation} 
\psi(q_{\mathrm{d}}) = -q_{\mathrm{d}} + \sum_{k=1}^{\infty} \left( (-1)^{3k+1} (6k+1) q_{\mathrm{d}}^{(6k+1)^2} + (-1)^{3k}(6k-1)q_{\mathrm{d}}^{(6k-1)^2} \right)\,.
\end{equation}
Here $q_{\mathrm{d}} = \exp(-\mathrm{area}(\Delta))$, with $\Delta$ the minimal triangle bounded by the Seidel Lagrangian.
\end{thm}

Consider the elliptic curve $E_{\rho}$ with $j(E_{\rho})=0$, it can be realized, say, by $x_{1}^3 + x_{2}^3 + x_{3}^3 = 0$ in $\mathbb{P}^{2}$. Its quotient\footnote{For example, this action could be realized as  $[x_{1},x_{2},x_{3}]\mapsto [ \exp( 2\pi i/3)  x_{1},  x_{2},  x_{3}]$ and should not be confused with the action of the group of 3-torsion points which moves the origin of the elliptic curve and thus is not an automorphism.} by the $\Z_3$ automorphism is $\bP^1_{3,3,3}$. The K\"ahler parameter $q$ of the elliptic curve is related with $q_{\mathrm{d}}$ by 
$q = q_{\mathrm{d}}^{24}$. Here the subscript `$d$' stands for `disk'. Throughout this paper, by abuse of notation, we will use for example the notation $\phi(q)$ to denote the quantity $\phi(q_{\mathrm{d}}(q))$.

An easy computation shows the following.

\begin{thm}
Both $\phi$ and $\psi$ (when expressed in $q$) are
modular forms of formal weight $3/2$ with the same multiplier system
for the modular group $\Gamma(3)$.
\end{thm}

\begin{proof}
Simple algebra shows that
\begin{align*}
\phi(q_{\mathrm{d}}) &= \frac{1}{2} \sum_{k=-\infty}^\infty (-1)^{k+1} (2k+1) q_{\mathrm{d}}^{9(2k+1)^2}= \sum_{r \in \Z+\frac{1}{2}} (-1)^{r+\frac{1}{2}} r q_{\mathrm{d}}^{36 r^2}.
\end{align*}
Recall that for the Jacobi theta function (here $v=\exp 2\pi i z$)
\begin{equation*}
 \theta_1(v,q) =  \sum_{r \in \Z+\frac{1}{2}} (-1)^{r} v^r q^{\frac{1}{2}r^2}\,,
\end{equation*}
we have
\begin{equation*}
 \partial_v|_{v=1} \theta_1(v,q) = \sum_{r \in \Z + \frac{1}{2}} (-1)^{r} r q^{\frac{1}{2}r^2}\,.
\end{equation*}
Therefore, we obtain
\begin{equation} 
\phi(q_{\mathrm{d}}) = {\bi} \partial_v \theta_1(1,q_{\mathrm{d}}^{72}) = {\bi} \partial_v \theta_1(1,q^3)\,.
\end{equation}
To compute $\psi$, we use the identity
\begin{equation*} 
\psi(q_{\mathrm{d}}) - 3 \phi(q_{\mathrm{d}}) = \sum_{l=0}^\infty (-1)^{l+1} (2l+1) q_{\mathrm{d}}^{(2l+1)^2} = {\bi} \partial_v \theta_1 (1,q_{\mathrm{d}}^8)\,.
\end{equation*}
Or alternatively,
$\phi(q_{\mathrm{d}}) -  \psi(q_{\mathrm{d}}^{9}) =-3\phi(q_{\mathrm{d}}^{9})$.
Hence
\begin{equation} \label{eq:psi333}
\psi(q_{\mathrm{d}}) = {\bi} (\partial_v \theta_1(1,q_{\mathrm{d}}^8) + 3 \partial_v \theta_1(1,q_{\mathrm{d}}^{72}))\,.
\end{equation}
where $q_{\mathrm{d}}^8 = q^{1/3}$ and $q_{\mathrm{d}}^{72}  = q^3$.
The Jacobi theta function $\theta_{1}$ satisfies
\begin{equation*}
2\pi\bi (v \partial_v)|_{v=1} \theta_1 = \partial_z|_{z=0} \theta_1 = -2\pi \eta(q)^3\,.
\end{equation*}
That is, $\partial_v|_{v=1} \theta_1 = \bi \eta(q)^3$.
It follows that $\phi(q),\psi(q)$ can be written in terms of the $\eta$--functions as follows:
\begin{align}
\phi(q) &= -\eta(q^3)^3\,,\\
\psi(q) &= -(\eta(q^{\frac{1}{3}})^3 + 3\eta(q^3)^3)\,.
\end{align}

\begin{remark}
Recall that 
for the Hesse-Dixon model for elliptic curves: $x^{3}+y^{3}+z^{3}-(\gamma+3)xyz=0$, we have  
\begin{equation}
\gamma(\tau)+3=3{A_{3}(\tau)\over C_{3}(\tau)}=3\left(1+ {\eta(q)^{12}\over 3^{3}\eta(q^3)^{12}}\right)^{1\over 3}\,,
\end{equation} 
where $\gamma$ is a Hauptmodul for the modular group $\Gamma(3)$, see for example \cite{Maier:2009} for details.  
The Hauptmodul is also called to be the mirror map since it gives a map between the K\"ahler moduli, parametrized by $\tau$, and the complex moduli parametrized by $\gamma$.
One can check that (see \cite{Borwein:1994})
\begin{equation}
{\psi (\tau)\over \phi (\tau)}=\gamma(\tau)+3\,.
\end{equation} 
That is, the generalized SYZ map is identical to the mirror map,
as has been deduced in Theorem 6.5 of \cite{CHL} in a different way.
This will be explained further in Section \ref{sectionexplanation}, where we see that actually the geometry defined by the open Gromov-Witten potential $W$ coincides with the Hesse-Dixon model.
\end{remark}

Using the results in Section \ref{sectionmodularforms}, we know that $\eta(q^{3})^{3}=3^{-{9\over 8}}B^{3\over 8}_{3}(\tau)C^{9\over 8}_{3}(\tau)$ is a modular form for $\Gamma_{0}(3)$ with possibly non-trivial multiplier system.
In particular, it is so for $\Gamma(3)$.
Therefore, this is also true for $\eta(q^{1\over 3})^{3}$ since
$\gamma(\tau)$ is modular with respect to $\Gamma(3)$ according to the above remark.
Moreover, $\phi,\psi$ must have the same multiplier system since $\gamma$ has a trivial one.
Hence the conclusion follows.

\end{proof}

\subsection{$(2,4,4)$ case}
\label{secevenodd}

\begin{thm}\cite{CHKL} \label{thm:W244}
The open Gromov-Witten potential of $\bP^1_{2,4,4}$ is
\begin{equation} \label{eq:W244}
W= q_{\mathrm{d}}^6 x^2 - q_{\mathrm{d}}xyz + d_y(q_{\mathrm{d}}) y^4 + d_z(q_{\mathrm{d}}) z^4 + d_{yz}(q_{\mathrm{d}}) y^2z^2,
\end{equation}
where
\begin{align}
d_y(q_{\mathrm{d}}) &= d_z(q_{\mathrm{d}}) = \sum_{0 \le r} (2r+1) q_{\mathrm{d}}^{16 (2r+1)^2-4}
+\sum_{0\le r < s} (2r+2s+2) q_{\mathrm{d}}^{16 (2r+1)(2s+1)-4}\,, \\
d_{yz}(q_{\mathrm{d}}) &= \sum_{r\ge1, s\ge1} \big( - (4r+4s-2) q_{\mathrm{d}}^{16(2r-1)2s-4} +  (2r+2s) q_{\mathrm{d}}^{64rs-4} \big)\,.
\end{align}
\end{thm}

The parameter $q_{\mathrm{d}}= \exp (-\mathrm{area}(\Delta))$, where $\Delta$ is the minimal disc bounded by the Seidel Lagrangian in $\bP^1_{2,4,4}$, is related to the K\"ahler parameter $q$ of the elliptic curve by $q = q_{\mathrm{d}}^{32}$.

We can rewrite $d_y$ and $d_{yz}$ in terms of the Eisenstein series $E_{2}(q)$ as follows.  First we recall that
\begin{equation}
\sum_{m,n \geq 1} m q^{mn} =\sum_{n=1}^{\infty} \sigma_{1}(n)q^{n} =\frac{1}{24} (1-E_2(q))\,.
\end{equation}
This identity implies that
\begin{align*}
\sum_{m,n \,\even}  m q^{mn} &= 2 \sum_{a,b} a q^{4ab} = \frac{1}{12} (1-E_2(q^4))\,.\\
\sum_{\substack{m \,\odd\\n \,\even}} m q^{mn} &= \sum_{\substack{m\\n \,\even}} m q^{mn} - \sum_{\substack{m \,\even\\n \,\even}} m q^{mn} = \frac{1}{24} (1-E_2(q^2)) - \frac{1}{12} (1-E_2(q^4))\\
&= \frac{1}{24} (-1 - E_2(q^2) + 2 E_2(q^4))\,.\\
\sum_{\substack{m \,\even\\n \,\odd}} m q^{mn} &= \sum_{\substack{m \,\even\\n}} m q^{mn} - \sum_{\substack{m \,\even\\n \,\even}} m q^{mn} = \frac{1}{12}(1-E_2(q^2)) - \frac{1}{12}(1-E_2(q^4)) \\
&= \frac{1}{12} (E_2(q^4) - E_2(q^2))\,.\\
\sum_{\substack{m \,\odd\\n \,\odd}} m q^{mn} &= \sum_{m,n} m q^{mn} - \sum_{\substack{m \,\odd\\n \,\even}} m q^{mn} - \sum_{\substack{m \,\even\\n \,\odd}} m q^{mn} - \sum_{\substack{m \,\even\\n \,\even}} m q^{mn} \\
&= - \frac{E_2(q)}{24} + \frac{1}{8} E_2 (q^2) - \frac{1}{12} E_2(q^4)\,.
\end{align*}
where the sums are all over positive integers.
Therefore, 
\begin{align}
d_{y}(q_{\mathrm{d}}) &= \frac{1}{2} \sum_{r,s\geq 0} (2r+2s+2) q_{\mathrm{d}}^{16(2r+1)(2s+1)-4}\nonumber\\
&= \frac{ 1}{2} q_{\mathrm{d}}^{-4}\sum_{m,n \,\odd} (m+n) q_{\mathrm{d}}^{16 mn} \nonumber \\
&= q_{\mathrm{d}}^{-4} \sum_{m,n \,\odd} m q_{\mathrm{d}}^{16 mn}  \nonumber\\
&= q^{-{1\over 8}} \left( - \frac{E_2(q^{\frac{1}{2}})}{24} + \frac{1}{8} E_2 (q) - \frac{1}{12} E_2(q^2) \right)\,,\label{eqnevenoddEisenstein}
\end{align}
and
\begin{align}
d_{yz}(q_{\mathrm{d}}) &= \sum_{r\ge1, s\ge1} \big( - (4r+4s-2) q_{\mathrm{d}}^{16(2r-1)2s-4} +  (2r+2s)  q_{\mathrm{d}}^{64rs-4} \big) \nonumber \\
&= -2 q_{\mathrm{d}}^{-4} \sum_{\substack{m \,\odd \\ n \,\even}} (m+n) q^{\frac{mn}{2}} + q_{\mathrm{d}}^{-4} \sum_{m, n \, \even} (m+n) q^{\frac{mn}{2}}  \nonumber\\
&= -2 q^{-{1\over 8}} \left( -\frac{1}{24} - \frac{E_2(q)}{8} + \frac{E_2(q^2)}{6} \right) + \frac{ 1}{6} q^{-{1\over 8}} (1-E_2(q^2)) \nonumber\\
&=q^{-{1\over 8}} \left( \frac{1}{4} + \frac{E_2(q)}{4} - \frac{E_2(q^2)}{2} \right)\,.
\end{align}

We now apply the results  for modular forms of the group $\Gamma_0(2)$  in Section \ref{sectionringofmodularforms}.
For this case, it is easy to see (for example, by dimension reasons) that
\begin{equation*}
A_{2}^2(q) = 2 E_2(q^2) - E_2(q)\,.
\end{equation*}
It is the generator for $M_{2}(\Gamma_{0}(2))$. Moreover, by using the $\eta$-expressions for the modular forms $A_{2},B_{2},C_{2}$, we get, see e.g., \cite{Maier:2011},
\begin{align*}
A_{2}^2 (q^2) &= \frac{1}{4} (A_{2}^2(q) + 3 B_{2}^2(q))\,, \\
C_{2}^2 (q^2) &= \frac{1}{4} (A_{2}^2(q) - B_{2}^2(q))\,,\\
A_{2}^4(q) &= B_{2}^4(q) +C_{2}^4(q)\,.
\end{align*}
Thus, we obtain
\begin{align}
d_y(q) &= \frac{q^{-{1\over 8}}}{24} \left(A_{2}^2(q^{\frac{1}{2}}) - A_{2}^2(q) \right) = \frac{1}{8} q^{-{1\over 8}}\cdot C_{2}^2(q)\,,\\
d_{yz}(q) &= q^{-{1\over 8}} \left( \frac{1}{4} - \frac{A_{2}^2(q)}{4} \right)\,.
\end{align}
Using the $\theta$-expansions for the modular forms of $N=2,4$ cases and the results on $M_{*}(\Gamma(4))$, we know that both $A_{2}^{2}=A_{4}^{2}+C_{4}^{2}$ and $C_{2}^{2}=2A_{4}C_{4}$
are modular forms of $\Gamma(4)$.
On can redefine the variables $x,y,z$ suitably to get rid of the constant $1/4$
and the multiplicative factor $q^{-{1\over 8}}$. Then the quantities $d_{y},d_{yz}$ become true modular forms.\\

Under the following change of coordinates in $(x,y,z)$
\begin{equation*}
x\mapsto q_{d}^{-3} (x+{q_{d}^{-2}\over 2} d_{y}^{-{1\over 4}} d_{z}^{-{1\over 4}}yz),\quad 
y\mapsto d_{y}^{-{1\over 4}} y,\quad
z\mapsto d_{z}^{-{1\over 4}}z\,,
\end{equation*}
the potential $W$ in \eqref{eq:W244} can be rewritten as
\begin{equation}\label{eq:244co}
 W=x^2 + y^4 + z^4 + \sigma(q_{\mathrm{d}}) y^2 z^2\,,
\end{equation}
where the generalized SYZ map is
\begin{equation} \label{sigma244}
\sigma(q_{\mathrm{d}}) := \frac{d_{yz}(q_{\mathrm{d}}) - (4 q_{\mathrm{d}}^4)^{-1}}{d_y(q_{\mathrm{d}})} = -\frac{2 A_{2}^2(q)}{C_{2}^2(q)}\,.
\end{equation}
Explicitly $\sigma(q_{\mathrm{d}})$ is the series
\begin{equation} \label{eq:sigma244}
\sigma(q_{\mathrm{d}}) = -\frac{1}{4 q_{\mathrm{d}}^{16}}-5 q_{\mathrm{d}}^{16}+\frac{31 q_{\mathrm{d}}^{48}}{2}-54  q_{\mathrm{d}}^{80}+\frac{641 q_{\mathrm{d}}^{112}}{4}-409  q_{\mathrm{d}}^{144}+\frac{1889 q_{\mathrm{d}}^{176}}{2}+\ldots
\end{equation}

We now show that $\sigma(q_{\mathrm{d}}(q))$, which comes from generating functions of polygon counting, is the inverse mirror map of the elliptic curve obtained by setting $W=0$ in \eqref{eq:W244} (again see Section \ref{sectionexplanation} for explanation).
We can express the inverse mirror map of the elliptic curve explicitly in terms of $\eta$-functions as follows. By the result on elliptic curve families of $E_{7}$ type in Section \ref{sectiongeometricmoduli}, the inverse mirror map (as the inverse of the map $a\mapsto \exp 2\pi i \tau(a)$) for
\begin{equation} \label{eq:244xyz}
x^2 + y^4 + z^4 + a xyz=0
\end{equation}
is 
\begin{equation} \label{eq:a244}
a(q) = 2^{3\over 2}\frac{A_{2}(q)}{C_{2}(q)}\,.
\end{equation}
To change \eqref{eq:244xyz} to the form of \eqref{eq:244co}, we replace $x$ by $x - \frac{a}{2} yz$ in \eqref{eq:244xyz} and obtain
\begin{equation} 
x^2 + y^4 + z^4 - \frac{a^2}{4} y^2z^2 =0\,,
\end{equation}
and so the inverse mirror map is
\begin{equation} 
 -\frac{a^2(q)}{4} = -\frac{2 A_{2}^2(q)}{C_{2}^2(q)}\,.
\end{equation}
This coincides with $\sigma(q_{\mathrm{d}}(q))$ in \eqref{sigma244}.
As a result, we conclude that

\begin{cor}
The generalized SYZ map equals to the inverse mirror map for $\bP^1_{2,4,4}$.
\end{cor}

\begin{remark}
We can express everything in terms of the Dedekind $\eta$-function
\begin{equation*}
 \eta(q) = q^{\frac{1}{24}} \left( 1 + \sum_{n=1}^\infty (-1)^n \left( q^{\frac{n(3n-1)}{2}} + q^{\frac{n(3n+1)}{2}} \right)\right)\,.
\end{equation*}
More precisely, from the $\eta$-expansions in Section \ref{sectionmodularforms}, we have
\begin{align*}
A_{2}(q) &= \frac{(2^6 \eta(q^2)^{24} + \eta(q)^{24})^{\frac{1}{4}}}{\eta(q)^2 \eta(q^2)^2}\,,\\
C_{2}(q) &= 2^{\frac{3}{2}} \frac{ \eta(q^2)^4}{\eta(q)^2}\,.
\end{align*}
Thus
\begin{equation}
 \sigma(q_{\mathrm{d}}(q)) = -2 \left( 1 + \frac{\eta(q)^{24}}{2^6 \eta(q^2)^{24}} \right)^{\frac{1}{2}}\,.
\end{equation}
\end{remark}

\subsection{$(2,3,6)$ case} \label{sec:236}

\begin{thm}\cite{CHKL} \label{thm:W236}
The open Gromov-Witten potential $W$ for $\bP^1_{2,3,6}$ is
\begin{equation}\label{eq:W236}
W=  q_{\mathrm{d}}^6 x^2 - q_{\mathrm{d}} xyz + c_y( q_{\mathrm{d}}) y^3 + c_z( q_{\mathrm{d}}) z^6  + c_{yz2}( q_{\mathrm{d}}) y^2z^2 + c_{yz4}( q_{\mathrm{d}}) yz^4\,,
\end{equation}
where
\begin{align}
A(n,a,b,c) &:=  {n+2\choose 2} - {a+1 \choose 2} - {b+1 \choose 2} - {c+1 \choose 2}, \label{eq:A}\\
c_y(q_{\mathrm{d}}) &= \sum_{a\ge0}(-1)^{a+1} (2a+1) q_{\mathrm{d}}^{48A(a-1,0,0,0) + 9}; \label{eq:cy}\\
c_{yz2}(q_{\mathrm{d}}) &= \sum_{n\ge a\ge0} \big(
(-1)^{n-a}(6n-2a+8) q_{\mathrm{d}}^{48A(n,a,0,0)-4} + (2n+4) q_{\mathrm{d}}^{48 A(n,a,n-a,0)- 4} \big); \label{eq:cyz2}\\
c_{yz4}(q_{\mathrm{d}}) &=
\sum_{a,b\ge0, n\ge a+b} (-1)^{n-a-b} (6n-2a-2b+7) q_{\mathrm{d}}^{48 A(n,a,b,0)-17};\label{cyz4} \\
c_z(q_{\mathrm{d}}) &= \sum (-1)^{n-a-b-c}\left({6n - 2a - 2b - 2c + 6}\over{\eta(n,a,b,c)}\right)
\cdot q_{\mathrm{d}}^{48 A(n,a,b,c)-30}\,. \label{eq:cz}
\end{align}
The summation in the expression of $c_z(q_{\mathrm{d}})$ is taken over
$(n,a,b,c)\in T_1\coprod T_2\coprod T_3\coprod T_6$,
\begin{align*}
T_6 =& \{(3a, a, a, a) \; : a\ge0\}, \\
T_3 =& \{(n,a,a,a)\;: n>3a\ge0\}, \\
T_2 =& \{(a+b+c,a,b,c) \;: a,b,c\ge0\text{ such that }a<\min(b,c)\text{ or }a=c< b \}, \\
T_1 =& \{ (a+b+c+k,a,b,c) \;: k \in \Z_{>0}, a,b,c\text{ are distinct non-negative integers such that } \\
& a<\min(b,c) \text{ or }a=c< b\},
\end{align*}
and $\eta(n,a,b,c)=i$ for $(n,a,b,c)\in T_i$.
\end{thm}

By the change of coordinates in $(x,y,z)$,
\begin{eqnarray*}
&&x\mapsto q_{d}^{-3}(x+ {1\over 2} q_{d}^{-2}c_{y}^{-{1\over 3}}s yz+  
{s^{3}(1-4 q_{d}^{4} c_{yz2}  )\over 24 q_{d}^{6} c_{y} } z^{3})\,,\\
&&y\mapsto c_{y}^{-{1\over 3}}(y+      s^{2} {1-4 q_{d}^{4} c_{yz2} \over 12 q_{d}^{4} c_{y}^{2\over 3} } z^{2})\,,\quad z\mapsto sz\,,
\end{eqnarray*}
where
\begin{equation*}
s= 864^{1\over 6} q_{d}^{2} c_{y}^{1\over 3}
(-1+12 q_{d}^{4} c_{yz2}   -48 q_{d}^{8} c_{yz2}^{2}
+72 q_{d}^{8} c_{y} c_{yz2}+64 q_{d}^{12} c_{yz2}^{3}
-288 q_{d}^{12}c_{y}c_{yz2} c_{yz4}+864 q_{d}^{12} c_{y}^{2}c_{z}
)^{-{1\over 6}}\,,
\end{equation*}
the open Gromov-Witten potential in \eqref{eq:W236} can be written as
\begin{equation}\label{eq:236co}
 x^2 + y^3 + z^6 + \sigma(q_{\mathrm{d}}) yz^4\,,
\end{equation}
where the generalized SYZ map is
\begin{multline} \label{eq:sigma236}
\sigma(q_{\mathrm{d}}) = \left(c_{yz4}(q_{\mathrm{d}}) - \frac{c_{yz2}^2(q_d)}{3 c_y(q_{\mathrm{d}})} - (48 q_d^8 c_y(q_{\mathrm{d}}))^{-1} + \frac{c_{yz2}(q_{\mathrm{d}})}{6 q_{\mathrm{d}}^4 c_y(q_{\mathrm{d}})} \right) c_y^{-\frac{1}{3}}(q_{\mathrm{d}})\\
\cdot \left(c_z(q_{\mathrm{d}}) + \frac{2 c_{yz2}^3(q_{\mathrm{d}})}{27 c_y^2(q_{\mathrm{d}})} - \frac{c_{yz2}(q_{\mathrm{d}}) c_{yz4}(q_{\mathrm{d}})}{3 c_y(q_{\mathrm{d}})}  - (864 q_d^{12} c_y^2(q_{\mathrm{d}}))^{-1} + \frac{c_{yz2}(q_{\mathrm{d}})}{72q_{\mathrm{d}}^8 c_y^2(q_{\mathrm{d}})} \right. \\
 \left.- \frac{c_{yz2}^2(q_{\mathrm{d}})}{18 q_{\mathrm{d}}^4 c_y^2(q_{\mathrm{d}})} + \frac{c_{yz4}(q_{\mathrm{d}})}{12 q_{\mathrm{d}}^4 c_y(q_{\mathrm{d}})} \right)^{-\frac{2}{3}}.
\end{multline}
By direct computation, $\sigma(q):=\sigma(q_{\mathrm{d}}(q))$ takes the form
\begin{equation} \label{eq:sigma}
\begin{aligned}
\sigma(q) =& - \frac{3}{2^{2/3}} \cdot (1+576 q+235008 q^2+109880064 q^3+53449592832 q^4 \\& +26574124961664 q^5+\ldots)
\end{aligned}
\end{equation}
and so $\sigma(q) = - \frac{3}{2^{2/3}}$ at $q=0$.

We now show that $\sigma(q_{\mathrm{d}}(q))$ is the inverse mirror map for the elliptic curve defined by setting $W$ in \eqref{eq:236co} to be zero, where $q = q_{\mathrm{d}}^{48}$. 
We also give an explicit expression of the inverse mirror map in terms of modular functions.  First, by the results in Section \ref{sectiongeometricmoduli} the inverse mirror map for
\begin{equation} \label{eq:236xyz}
x^2 + y^3 + z^6 + a xyz
\end{equation}
is
\begin{equation} \label{eq:a}
a = -(432)^{\frac{1}{6}} \cdot \frac{E_4^{\frac{1}{4}}}{((E_4^{\frac{3}{2}} - E_6)/2)^{\frac{1}{6}}}\,.
\end{equation}
where $E_4$ and $E_6$ are the Eisenstein series.
Again as before we are now considering the elliptic curve family given by $W=0$.
Then we apply a change of coordinates in $(x,y,z)$ to change \eqref{eq:236xyz} to the form in \eqref{eq:236co}.  This is achieved by first replacing $x$ by $x - \frac{a}{2} yz$ to change the term $xyz$ to $y^2z^2$, and then replacing $y$ to $y+\frac{a^2}{12} z^2$ to replace the term $y^2z^2$ to $yz^4$.  As a result, \eqref{eq:236xyz} is changed to
\begin{equation}
x^2 + y^3 + z^6 - \frac{3 a^4}{2^{\frac{2}{3}} (864 - a^6)^{\frac{2}{3}}} yz^4\,.
\end{equation}
By substituting $a$ in \eqref{eq:a} into the above expression, we obtain that the inverse mirror map for the elliptic curve $x^2 + y^3 + z^6 + s yz^4$ given by
\begin{equation}\label{eqmirrormap236}
s(q) = \frac{-3 E_4^3(q)}{2^{\frac{2}{3}} E_6^2(q)}\,.
\end{equation}
One can do a computational check that $s(q)$ has the same expression in \eqref{eq:sigma} as $\sigma(q)$.

\begin{remark}
Similar to the other cases, we expect the quantities $c_{y},c_{yz2},c_{yz4},c_{z}$ to be modular forms up to addition and multiplication by some factors which are not essential, so that the generalized SYZ map in \eqref{eq:sigma236} coincides with the expression given in \eqref{eqmirrormap236}.
This is true for $c_{y}$. In fact, we have 
\begin{eqnarray}
c_{y}(q_{\mathrm{d}})&=&
q_{\mathrm{d}}^{9}
\sum_{a\geq 0} (-1)^{a+1}(2a+1)q_{\mathrm{d}}^{24a(a+1)} \nonumber\\
&=&q_{\mathrm{d}}^{3}
\sum_{a\geq 0} (-1)^{a+1}(2a+1)q_{\mathrm{d}}^{24(a+{1\over 2})^{2}} \nonumber\\
&=&2q^{1\over 16}
\sum_{r\geq 0, r\in \mathbb{Z}+{1\over 2}} (-1)^{r+{1\over 2}}rq^{{1\over 2}r^{2}} \nonumber\\
&=&q^{1\over 16}\bi \partial_{v}\theta_{1}|_{v=1} \nonumber\\
&=&-q^{1\over 16}\eta(q)^{3}\,.
\end{eqnarray}
Also the second term in $c_{yz2}$ (which counts parallelograms) is
\begin{eqnarray}
&&q_{\mathrm{d}}^{-4}\sum_{n\geq a\geq 0}(2n+4)q_{\mathrm{d}}^{48(a+1)(n-a+1)} \nonumber\\
&=&q_{\mathrm{d}}^{-4}\sum_{a\geq 0,b\geq 0}(2(a+b)+4)q_{\mathrm{d}}^{48(a+1)(b+1)} \nonumber\\
&=&2q_{\mathrm{d}}^{-4}\sum_{a\geq 1,b\geq 1}(a+b)q_{\mathrm{d}}^{48ab} \nonumber\\
&=& 2q_{\mathrm{d}}^{-4} \cdot {1\over 12}(1-E_{2}(q_{\mathrm{d}}^{48})) \nonumber\\
&=& {1\over 6}q^{-{1\over 12}} (1-E_{2}(q))\,.
\end{eqnarray}
We conjecture that the rest are quasi-modular forms as introduced by \cite{Kaneko:1995} and the overall coefficients are modular forms. See Section \ref{secmf236} for further discussions.
\end{remark}

\subsection{$(2,2,2,2)$ case} \label{sec:2222}
The remaining case of elliptic orbifolds is $\bP^1_{2,2,2,2}$. It can be constructed as a quotient of an elliptic curve $E$ by $\Z_2$, where $1 \in \Z_2$ acts by $z \mapsto -z \in E$.  The generalized SYZ mirror construction in this case is rather different, namely it involves more than one reference Lagrangians.  The construction is given in \cite{CHLnc}, here we quote the result below.  It turns out that the mirror is not an isolated singularity, and hence Saito's theory of primitive forms does not apply directly to this case.

\begin{thm}\cite{CHLnc} \label{thm:W2222}
The open Gromov-Witten potential of $\bP^1_{2,2,2,2}$ is
\begin{equation} \label{eq:W2222}
W=\phi(q_d) ((xy)^2 + (xw)^2 + (zy)^2 + (zw)^2) + \psi(q_d) xyzw
\end{equation}
defined on the resolved conifold $\CO_{\bP^1}(-1) \oplus \CO_{\bP^1}(-1) = (\C^4-Z) / \C^\times$, where $(x,y,z,w)$ are the standard coordinates of $\C^4$, $Z = \{x = z = 0\}$, $\C^\times$ acts by $\lambda \cdot (x,y,z,w) = (\lambda x, \lambda^{-1} y, \lambda z, \lambda^{-1} w)$, and
\begin{eqnarray*}
\phi(q_{d})&=&\sum_{k,l\geq 0}^{\infty}(4k+1)q_{d}^{(4k+1)(4l+1)}+\sum_{k,l\geq 0}^{\infty}(4k+3)q_{d}^{(4k+3)(4l+3)}\,,\\
\psi(q_{d})&=&
\sum_{k,l\geq 0}^{\infty}(k+l+1)q_{d}^{(4k+1)(4l+3)}\,.
\end{eqnarray*}
\end{thm}

The parameter $q_{\mathrm{d}}= \exp (-\mathrm{area}(\Delta))$, where $\Delta$ is a certain holomorphic square in $\bP^1_{2,2,2,2}$, is related to the K\"ahler parameter $q$ of the elliptic curve $E$ by $q = q_{\mathrm{d}}^{8}$.

By direct computation, the critical locus of $W$ is the zero section $\bP^1 \subset \CO_{\bP^1}(-1) \oplus \CO_{\bP^1}(-1)$ instead of a point.  The Frobenius structure on the universal deformation space of $W$ is unclear since Saito's theory is not yet known for non-isolated singularities.  Nevertheless, we can consider the mirror elliptic curve family to obtain the flat coordinate for marginal deformations, and compare it with the generalized SYZ map $\psi / \phi$.  

To be more precise, $W$ descends to the quotient of $\CO_{\bP^1}(-1) \oplus \CO_{\bP^1}(-1)$ by $\Z_2$, which is the total space of canonical line bundle $K_{\bP^1 \times \bP^1}$.  The critical locus of $W$ in $K_{\bP^1 \times \bP^1}$ is the elliptic curve $\{W=0\} \subset \bP^1 \times \bP^1$ which is the mirror of $E$, where $(x:z,y:w)$ are the standard homogeneous coordinates on $\bP^1 \times \bP^1$.  It can also be embedded into $\bP^3$ via Segre embedding
$$x_1 = xy, x_2 = xw, x_3 = zw, x_4 = zy.$$
Then the mirror of $E$ is the elliptic curve given as the complete intersection
$$\{x_1 x_3 = x_2 x_4\} \cap \{\phi(q_d) (x_1^2 + x_2^2 +x_3^2 +x_4^2) + \psi(q_d) x_1 x_3 = 0\} \subset \bP^3.$$

The $j$-invariant of the elliptic curve family
$$\{((xy)^2 + (xw)^2 + (zy)^2 + (zw)^2) + \sigma xyzw=0\} \subset \bP^1 \times \bP^1$$ 
can be obtained by using the algorithm provided in \cite{Connell:1996}, which is 
\begin{equation}
j(\sigma)=\frac{\left(\sigma^4-16 \sigma^2+256\right)^3}{\sigma^4 \left(\sigma^2-16\right)^2}\,.
\end{equation}
Comparing this with the $j$-invariant for the $E_{5}$ elliptic curve family discussed 
in Section \ref{sectiongeometricmoduli}, we are led to
\begin{equation}\label{eqnmodular2222}
\sigma=2\cdot {1+ \alpha^{1\over 2} \over \alpha^{1\over 4}}\,,
\end{equation}
where $\alpha$ is the Hauptmodul for $\Gamma_{0}(4)$. 

Now we consider the generalized SYZ map $\psi/\phi$.
We can rewrite $\phi(q_{d})$ and $\psi(q_{d})$ in terms of $\eta$-products as follows.
Using the computations used in deriving \eqref{eqnevenoddEisenstein}, we find
\begin{equation}
\psi(q_{d})+4\phi(q_{d})={\eta(q_{d}^{4})^{8}\over \eta(q_{d}^{2})^{4}}\,.
\end{equation}
This identity implies that
\begin{equation}
\psi(q_{d})-4\phi(q_{d})={\eta(q_{d}^{4})^{8}\over \eta(-q_{d}^{2})^{4}}
={\eta(q_{d}^{8})^{4}\eta(q_{d}^{2})^{4}\over \eta(q_{d}^4)^{4}}\,.
\end{equation}
Solving for $\phi(q_{d}),\psi(q_{d})$ from the above two identities, we obtain
\begin{equation}
\phi(q_{d})={\eta(q_{d}^{8})^{2}\eta(q_{d}^{16})^{4}\over \eta(q_{d}^{4})^{2}}\,,\quad 
\psi(q_{d})={\eta(q_{d}^{8})^{14}\over\eta(q_{d}^{4})^{6}\eta(q_{d}^{16})^{4}}\,.
\end{equation}
Now using the $\eta$-expansions of the modular forms for $\Gamma_{0}(4)$ in Table \ref{tableetaexpansions}, we get (recall $q_{d}=q^{1\over 8}$)
\begin{equation}
\phi={1\over 2^{3}}A_{4}(q^{1\over 2})^{1\over 2}C_{4}(q^{1\over 2})^{3\over 2}\,,\quad
\psi={1\over 2}A_{4}(q^{1\over 2})^{1\over 2}C_{4}(q^{1\over 2})^{1\over 2}\,.
\end{equation}
Since $\Gamma_{0}(4)$ is isomorphic to $\Gamma(2)$ via $\tau\mapsto 2\tau$, 
we know that if $f(\tau)$ is a modular form for $\Gamma_{0}(4)$, then $f({\tau\over 2})$
is so for $\Gamma(2)$. This tells that $\phi,\psi$
are modular forms for $\Gamma(2)$.\\

It follows that the generalized SYZ map is 
\begin{equation}\label{eqnmirrormap2222}
{\psi(q_{d} (q) ) \over \phi(q_{d}(q))}
=4 {A_{4}(q^{1\over })\over C_{4}(q^{1\over 2})}
={\eta(q)^{12}\over \eta(q^{2})^{8}\eta(q^{1\over 2})^{4}}\,.
\end{equation}

Using the $\eta$-expansions of the modular forms for $\Gamma_{0}(4)$ in Table \ref{tableetaexpansions}, we see that the generalized SYZ map in \eqref{eqnmirrormap2222} produced by Lagrangian Floer theory is identical to the modular function given by \eqref{eqnmodular2222}.  As a result, the generalized SYZ map equals to the inverse mirror map for $\bP^1_{2,2,2,2}$.


\section{Modularity of matrix factorizations}
\label{sectionmatrixfactorizations}

In \cite{CHL}, an $A_\infty$ functor was constructed from the Fukaya category of Lagrangian branes in a symplectic manifold $X$ to the category of matrix factorizations of the open Gromov-Witten potential $W$. The construction of $W$ was reviewed in the beginning of Section \ref{sectionGWpotentials}.  For $W \in R = \C[z_1,\ldots,z_n]$, a matrix factorization is simply an odd endomorphism $\delta$ on a $\Z_2$-graded $R$-module $M = M_0 \oplus M_1$ which satisfies $\delta^2 = W \cdot \mathrm{Id}$.  Such a functor is motivated from the celebrated homological mirror symmetry conjecture \cite{kontsevich94}.

Let us review very briefly the functor in the object level.  Given a spin oriented Lagrangian $L$ which intersects the reference Lagrangian $\BL$ (fixed in the beginning of Section \ref{sectionGWpotentials}) transversely, define $M = \oplus_p R \cdot p$ where the sum is over all intersection points $p \in L \cap \BL$, and $R \cdot p$ has odd (or even) degree if $p$ has odd (or even) degree.  Then $\delta$ is defined to be $m_1^{(b,0)}$ (which automatically has odd degree), which is roughly speaking counting pseudoholomorphic strips with one side bounded by $(\BL,b)$ and another side bounded by $L$.  Since the formal deformation $b$ is assumed to be weakly unobstructed, it follows from the $A_\infty$ relation
\begin{equation}
 (m_1^{(b,0)})^2 = m_2(m_0^b, \cdot ) = m_2(m_0^b, \cdot ) = m_2(W(b) \one_{\tilde{\BL}}, \cdot ) = W(b) \cdot \mathrm{Id} 
\end{equation}
that $\delta$ is a matrix factorization.  

In particular, the Seidel Lagrangian of an elliptic $\bP^1$ orbifold can be transformed to a matrix factorization of the open Gromov-Witten potential $W$.  They are split generators of the derived Fukaya category and the derived category of matrix factorizations respectively.
In this section, we study the modularity of the matrix factorizations constructed from the potential $W$ for the elliptic orbifolds.

\subsection{$(3,3,3)$ case}

The matrix factorization mirror to the Seidel Lagrangian in $\bP^1_{3,3,3}$ was computed in \cite[Section 7.7]{CHKL}.  In the following we check that their coefficients are modular forms with possibly non-trivial multiplier systems.

\begin{thm} \label{thm:MF333}
The matrix factorization mirror to the Seidel Lagrangian in $\bP^1_{3,3,3}$ is $M=(\largewedge^* \C^3, \delta)$ where $ \delta = (x X + y Y + z Z) \wedge \cdot + w_x \iota_X + w_y \iota_Y + w_z \iota_Z, $
and $w_x, w_y, w_z$ are the following polynomials whose coefficients are modular forms:
$$w_x = (-\eta(q^3)^3) x^2 + \left( -\frac{1}{3} \eta(q^{\frac{1}{3}})^3 + \eta(q^3)^3 - \frac{2}{3} \eta(q) \right) yz,$$
$$w_y = (-\eta(q^3)^3) y^2 + \left( -\frac{1}{3} \eta(q^{\frac{1}{3}})^3 + \eta(q^3)^3 + \frac{1}{3} \eta(q) \right) xz,$$
$$w_z = (-\eta(q^3)^3) z^2 + \left( -\frac{1}{3} \eta(q^{\frac{1}{3}})^3 + \eta(q^3)^3 + \frac{1}{3} \eta(q) \right) xy.$$
\end{thm}

\begin{proof}
From the result of \cite[Section 7.7]{CHKL}, the matrix factorization is $(M,\delta)$ defined above where
\begin{align*}
w_x =&   x^2 \sum_{k=0}^{\infty}  (-1)^{k+1}(2k+1) q_{\mathrm{d}}^{(3(2k+1))^2} \\
&+yz \left( -q_{\mathrm{d}} + \sum_{k=1}^{\infty} (-1)^{k+1} \left(  (2k+1) q_{\mathrm{d}}^{(6k+1)^2} - (2k -1) q_{\mathrm{d}}^{(6k-1)^2} \right) \right), \\
w_y =& y^2 \sum_{k=0}^{\infty}  (-1)^{k+1} (2k+1) q_{\mathrm{d}}^{(3(2k+1))^2} + xz \sum_{k=1}^{\infty}  (-1)^{k+1} \left(  2k q_{\mathrm{d}}^{(6k+1)^2} -  2k q_{\mathrm{d}}^{(6k-1)^2} \right), \\
w_z =& z^2 \sum_{k=0}^{\infty}  (-1)^{k+1} (2k+1) q_{\mathrm{d}}^{(3(2k+1))^2} +  xy \sum_{k=1}^{\infty}  (-1)^{k+1} \left(  2k q_{\mathrm{d}}^{(6k+1)^2} -  2k q_{\mathrm{d}}^{(6k-1)^2} \right). 
\end{align*}
The coefficient of $x^2$ in $w_x$ (or that of $y^2$ in $w_y$, or that of $z^2$ in $w_z$) equals to $\phi(q_{\mathrm{d}}) = \bi \partial_v \theta_1 (q^3,1)$.  The coefficient of $yz$ in $w_x$ is
\begin{align}
&-q_{\mathrm{d}} + \sum_{k=1}^{\infty} (-1)^{k+1} \left(  (2k+1) q_{\mathrm{d}}^{(6k+1)^2} - (2k-1) q_{\mathrm{d}}^{(6k-1)^2} \right)  \nonumber\\
=& \sum_{k=-\infty}^\infty (-1)^{k+1} (2k+1) q_{\mathrm{d}}^{(6k+1)^2} \nonumber\\
=& \frac{1}{3} \sum_{k=-\infty}^\infty (-1)^{k+1} (6k+3) q_{\mathrm{d}}^{(6k+1)^2}\nonumber \\
=& \frac{\psi(q_{\mathrm{d}})}{3} + \frac{2}{3} \sum_{k=-\infty}^\infty (-1)^{k+1} q_{\mathrm{d}}^{(6k+1)^2}\nonumber\\
=& \frac{1}{3} \psi(q_{\mathrm{d}}) - \frac{2}{3} \eta(q_{\mathrm{d}}^{24})\,,
\end{align}
where we have used the identity that
\begin{equation*}
\eta(q) = q^{\frac{1}{24}} \sum_{k=-\infty}^\infty (-1)^k q^{\frac{3k^2-k}{2}}\,.
\end{equation*}
Written in terms of the parameter $q$, this is
\begin{equation}
{1\over 3}\psi(q)-{2\over 3}\eta(q)\,.
\end{equation}
The coefficient of $xz$ in $w_y$ (or that of $xy$ in $w_z$) is
\begin{align}
&\sum_{k=1}^{\infty}  (-1)^{k+1} \left(  2k q_{\mathrm{d}}^{(6k+1)^2} -  2k q_{\mathrm{d}}^{(6k-1)^2} \right)  \nonumber\\
=& \sum_{k=-\infty}^{\infty} (-1)^{k+1} (2k) \cdot q_{\mathrm{d}}^{(6k+1)^2} \nonumber\\
=& \frac{1}{3} \sum_{k=-\infty}^{\infty} (-1)^{k+1} (6k+1) q_{\mathrm{d}}^{(6k+1)^2} - \frac{1}{3} \sum_{k=-\infty}^{\infty} (-1)^{k+1} q_{\mathrm{d}}^{(6k+1)^2} \nonumber\\
=& \frac{1}{3} \psi(q_{\mathrm{d}})+ \frac{1}{3} \eta(q_{\mathrm{d}}^{24}).
\end{align}
Written in terms of the parameter $q$, this is
\begin{equation}
{1\over 3}\psi(q)+{1\over 3}\eta(q)\,.
\end{equation}
All mentioned earlier in Section \ref{sectionGWpotentials}, both $\phi,\psi$ are modular forms with respect to $\Gamma(3)$, hence all the coefficients studied here are modular forms, and they have the explicit expressions as stated in the theorem.
\end{proof}

\begin{remark}
It is easy to check that $x w_x + y w_y + z w_z = W$ by straightforward calculation.
\end{remark}

\subsection{$(2,4,4)$ case}

\begin{thm} \label{thm:MF244}
The matrix factorization mirror to the Seidel Lagrangian of $\bP^1_{2,4,4}$ is $M=(\largewedge^* \C^3, \delta)$ where
$$ \delta = (x X + y Y + z Z) \wedge \cdot + w_x \iota_X + w_y \iota_Y + w_z \iota_Z, $$
and $w_x, w_y, w_z$ are the following polynomials whose coefficients are modular forms (up to a multiple by a power of $q$):
\begin{align*}
w_x &= q^{\frac{3}{16}} x - q^{\frac{1}{32}} yz,\\
w_y &= \left(\frac{1}{8} q^{-{1\over 8}}\cdot C_{2}^2(q)\right) y^3 + \left( \frac{q^{-{1\over 8}}}{8} \left( 1 - A_{2}^2(q) \right) \right) yz^2,\\
w_z &= \left(\frac{1}{8} q^{-{1\over 8}}\cdot C_{2}^2(q)\right) z^3 + \left( \frac{q^{-{1\over 8}}}{8} \left( 1 - A_{2}^2(q) \right) \right) y^2z.
\end{align*}
\end{thm}
\begin{proof}
It is a direct computation as in \cite[Section 7.7]{CHKL} that the mirror matrix factorization is $(M,\delta)$ defined above, where
\begin{align*}
w_x =& q_{\mathrm{d}}^6 x - q_{\mathrm{d}} yz\,,\\
w_y =& \left(\sum_{0 \le r} (2r+1) q_{\mathrm{d}}^{16 (2r+1)^2-4}
+\sum_{0\le r < s} (2r+2s+2) q_{\mathrm{d}}^{16 (2r+1)(2s+1)-4}\right) y^3 \\
&+ \left(\sum_{r\ge1, s\ge1} \big(- (2r+2s-1) q_{\mathrm{d}}^{16(2r-1)2s-4} +  2r q_{\mathrm{d}}^{64rs-4} \big)\right) yz^2\,,\\
w_z =& \left(\sum_{0 \le r} (2r+1) q_{\mathrm{d}}^{16 (2r+1)^2-4}
+\sum_{0\le r < s} (2r+2s+2) q_{\mathrm{d}}^{16 (2r+1)(2s+1)-4}\right) z^3 \\
&+ \left(\sum_{r\ge1, s\ge1} \big(- (2r+2s-1) q_{\mathrm{d}}^{16(2r-1)2s-4} +  2s q_{\mathrm{d}}^{64rs-4} \big)\right) y^2z\,.
\end{align*}
The coefficient of $y^3$ in $w_y$ (or that of $z^3$ in $w_z$) is nothing but $d_y$ studied in Section \ref{sectionGWpotentials}, while the coefficient of $yz^2$ in $w_y$ (or that of $y^2z$ in $w_z$) is $d_{yz}/2$. They have been shown to be modular forms with respect to $\Gamma(4)$ in the previous section.
\end{proof}

\subsection{$(2,3,6)$ case}
\label{secmf236}

Similarly, we can directly compute the matrix factorization mirror to the Seidel Lagrangian of $\bP^1_{2,3,6}$.  The result is $(M=\largewedge^* \C^3, \delta)$, where
$$ \delta = (x X + y Y + z Z) \wedge \cdot + w_x \iota_X + w_y \iota_Y + w_z \iota_Z, $$
and $w_x, w_y, w_z$ are defined by
\begin{align*}
w_x =& q_{\mathrm{d}}^6 x - q_{\mathrm{d}} yz,\\
w_y =& c_y(q_{\mathrm{d}}) y^2 + y z^2 \sum_{a,b\ge 0} \big(
(-1)^{b}(2a+4b+5) q_{\mathrm{d}}^{48A(a+b,a,0,0)-4} + (2b+2) q_{\mathrm{d}}^{48 A(a+b,a,b,0)- 4} \big) \\
&+ z^4 \sum_{a,b\ge0, n\ge a+b} (-1)^{n-a-b} (2n-2a+2) q_{\mathrm{d}}^{48 A(n,a,b,0)-17},\\
w_z =& c_z(q_{\mathrm{d}}) z^5 + y^2 z \sum_{a,b\ge 0} \big(
(-1)^{b}(2a+2b+3) q_{\mathrm{d}}^{48A(a+b,a,0,0)-4} + (2a+2) q_{\mathrm{d}}^{48 A(a+b,a,b,0)- 4} \big) \\
&+ yz^3 \sum_{a,b\ge0, n\ge a+b} (-1)^{n-a-b} (4n-2b+5) q_{\mathrm{d}}^{48 A(n,a,b,0)-17},
\end{align*}
and $A(n,a,b,c)$, $c_y$ and $c_z$ are given in \eqref{eq:A}, \eqref{eq:cy} and \eqref{eq:cz} respectively.  \\

The sum of coefficients for the $yz^{2},y^{2}z$ terms of $(M,\delta)$ gives the one for $y^{2}z^{2}$ in $W$, similarly for $z^{4},yz^{3}$ terms.
Recall that 
\begin{equation*}
q_{\mathrm{d}}^{48}=q\,,\quad A(n,a,b,c) =  {n+2\choose 2} - {a+1 \choose 2} - {b+1 \choose 2} - {c+1 \choose 2}\,.
\end{equation*} 
By pulling out $q_{\mathrm{d}}^{-4}$ for the first parts in the $yz^{2}, y^{2}z$ terms, we get
\begin{eqnarray}
&&\sum_{a,b\geq 0} (-1)^{b} (4b+2a+5)q^{{1\over 2} (b+1))b+1+2a+1)}\,,\label{eq2361}\\
&&\sum_{a,b\geq 0} (-1)^{b} (2b+2a+3)q^{{1\over 2} (b+1)(b+1+2a+1)}\,.\label{eq2362}
\end{eqnarray}
The following quantity is easily computed:
\begin{equation}\label{eq2363}
\sum_{a,b\geq 0} (-1)^{b}(2a+1)q^{{1\over 2} (b+1)(b+1+2a+1)}={1\over 24}(1-E_{2}(q))\,.
\end{equation}
More precisely, we have 
\begin{eqnarray*}
&&\sum_{a,b\geq 0} (-1)^{b}(2a+1)q^{{1\over 2} (b+1)(b+1+2a+1)}\\
&=& \sum_{
\substack{k\geq 1,\,l\geq k,\\
 l=k+\mathrm{odd}}} (-1)^{k-1}(l-k)q^{{1\over 2} kl}\\
&=& \sum_{
\substack{k\geq 1,\,l\geq k,\\
 l=k+\mathrm{odd}}} ((-1)^k k+(-1)^l l)q^{{1\over 2} kl}\\
&=& \sum_{
\substack{k,l\geq 1,\\
 l=k+\mathrm{odd}}} (-1)^k k q^{{1\over 2} kl}\\
&=& \sum_{
\substack{k,\,l\geq 1,\\
 k=\mathrm{odd},\\ l=\mathrm{even}}} (-1)^k k q^{{1\over 2} kl}
+
\sum_{
\substack{k,\,l\geq 1,\\
 k=\mathrm{even},\\ l=\mathrm{odd}}} (-1)^k k q^{{1\over 2} kl}
\\
&=&- \sum_{
\substack{k,\,l\geq 1,\\
 k=\mathrm{odd},\\ l=\mathrm{even}}}  k q^{{1\over 2} kl}
+
\sum_{
\substack{k,\,l\geq 1,\\
 k=\mathrm{even},\\ l=\mathrm{odd}}}  k q^{{1\over 2} kl}
\,.
\end{eqnarray*}
Then the statement follows from the summations we computed in Section \ref{secevenodd}.

Comparing \eqref{eq2361}, \eqref{eq2362} with \eqref{eq2363}, we can see what is left is to calculate
\begin{equation*}
\sum_{a,b\geq 0} (-1)^{b}(b+1)q^{{1\over 2} (b+1)(b+1+2a+1)}\,.
\end{equation*}
This can be simplified further as follows (changing the variable $b+1$ to $k$)
\begin{equation}\label{eq2364}
\sum_{k\geq 1,a\geq 0} (-1)^{k-1}kq^{{1\over 2} k(k+2a+1)}=\sum_{k\geq 1} (-1)^{k-1}k {q^{{ 1 \over 2}( k^{2}+k) }\over 1-q^{k}}\,.
\end{equation}
It is related to the derivative of the Appell function of level one.
The other terms involving $2b+2,2a+2$ in the $yz^{2}, y^{2}z$ terms can be calculated due to symmetry and the result for $W$, both are equal to $q^{-{1\over 12}}(1-E_{2}(q))/12$.
For the coefficient of $z^{4}$ in $w_{y}$ and that of $yz^{3}$ in $w_{z}$, we need to compute (by pulling out $q_{\mathrm{d}}^{-17}$, using $q_{\mathrm{d}}^{48}=q$ and defining $k=n-a-b$)
\begin{eqnarray*}
&&\sum_{k,a,b\geq 0} (-1)^{k}(2k+2b+2)q^{1+a+b+a b+\frac{3 k}{2}+a k+b k+\frac{k^2}{2} }\,,\\
&&\sum_{k,a,b\geq 0} (-1)^{k} (4k+4a+2b+5)q^{1+a+b+a b+\frac{3 k}{2}+a k+b k+\frac{k^2}{2} }\,.
\end{eqnarray*}
Taking the difference of the above two formulas, and simplifying a little further, we are left with
\begin{eqnarray}
&&\sum_{k,a,b\geq 0} (-1)^{k}(2a+1)q^{1+a+b+a b+\frac{3 k}{2}+a k+b k+\frac{k^2}{2} }\,, \label{eq2365} \\
&&\sum_{k,a,b\geq 0} (-1)^{k} (2k+1)q^{1+a+b+a b+\frac{3 k}{2}+a k+b k+\frac{k^2}{2} }\,. \label{eq2366}
\end{eqnarray}
We expect that all the quantities in \eqref{eq2364}, \eqref{eq2365}, \eqref{eq2366} are quasi-modular forms (up to a multiple of a power of $q$) for $\Gamma(6)$ with possibly non-trivial multiplier systems. This would then imply that the coefficients in the matrix factorization $(M,\delta)$ for the $(2,3,6)$ case are modular. However, we are not able to prove this at this moment.\footnote{We are kindly informed by Kathrin Bringmann and Larry Rolen in a private communication that these summations are nice objects which are related to mock modular forms.}


\section{Mirror symmetry over global moduli}
\label{sectionexplanation}

In Section \ref{sectionGWpotentials} and Section \ref{sectionmatrixfactorizations} we proved that the potential $W$ and the matrix factorization $M$
are modular for some modular group $\Gamma$ which depends on the geometry, hence they extend automatically to be sections of holomorphic line bundles on the modular curves $\Gamma\backslash \mathcal{H}^{*}$. The proof is based on straightforward calculations.  In this section we explain why modularity is expected from the point of view of global mirror symmetry.

\subsection{LG/CY correspondence}

It is well-known that the elliptic curve is self-mirror.  This simple important fact can be obtained using group action and LG/CY correspondence as follows.

Given a symplectic torus $E$, we equip it with the complex structure with an automorphism group $G$, where $G = \Z_3,\Z_4$ or $\Z_6$.  Then $E/G = \bP^1_{3,3,3}, \bP^1_{2,4,4}$ or $\bP^1_{2,3,6}$ respectively.  By the mirror construction \cite{CHL} which is briefly explained in the beginning of Section \ref{sectionGWpotentials}, the Landau-Ginzburg mirror is the open Gromov-Witten potential $W$ defined on $\C^3$ whose explicit expressions are given in Theorem \ref{thm:W333}, \ref{thm:W244} or \ref{thm:W236} respectively. The potential $W$ is invariant under the action of the dual group $\check{G} \cong G$, and the mirror of $E$ is given by $(\C^3/\check{G}, W)$ \cite{Seidel:g=2,CHL}.  By LG/CY correspondence \cite{Orlov}, the complex geometry (so-called the B-model) of $(\C^3/\check{G}, W)$ is equivalent to that of the elliptic curve $\check{E} = \{W= 0\} \subset W\bP^2$, where $W\bP^2$ is the weighted projective space $(\C^3-\{0\}) / \C^\times$ and the $\C^\times$ action has weights $(1,1,1)$, $(1,2,2)$ and $(2,3,6)$ respectively.  This gives an explanation, which is different from the usual SYZ approach, of why the elliptic curve is self-mirror.

\begin{figure}[h]
\centering
   \renewcommand{\arraystretch}{1.2} 
\begin{displaymath}																		
\xymatrix{
\textrm{A-model} &&\ar@{->} [l] \textrm{mirror symmetry}\ar@{->}[r] && \textrm{B-model}\\
E/G  \ar@/_/  [ddrrrr]^{\hspace{30pt}\textrm{Open Gromov-Witten potential}}&& \ar@{->} [l] \textrm{mirror}\ar@{->}[r] &&  \check{E} =  \{W=0\} \subset W\bP^2 \\
  &&&&  \\
&&&& (\mathbb{C}^{3}/\check{G},W)
\ar[uu]_{\textrm{LG/CY correspondence}}
\\
}
\end{displaymath}	
   \caption[Chain of dualities]{Chain of dualities}
  \label{chainofdualities}																	
\end{figure}
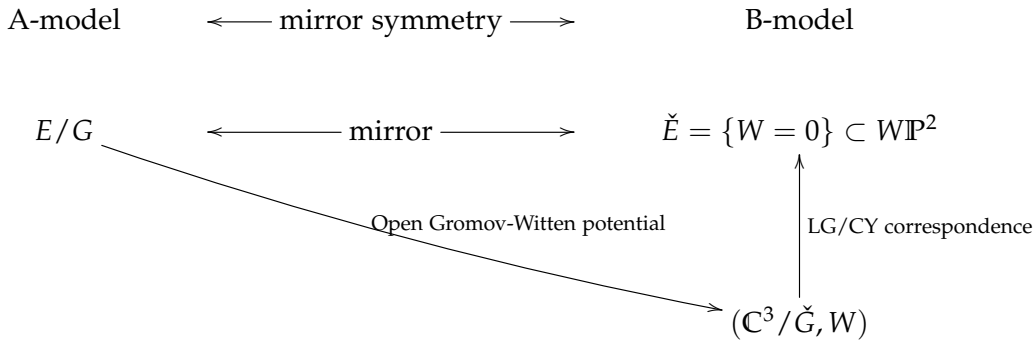

The moduli space of complex structures on $\check{E}$ is the (compactified) upper half plane quotient by $\mathrm{SL}(2,\Z)$.  By global mirror symmetry, the K\"ahler moduli of $E$ is also the upper half plane quotient by $\mathrm{SL}(2,\Z)$ (this can also be seen from considering the moduli space of Bridgeland stability conditions \cite{Bridgeland}).  The global mirror map in this case is simply given by the identity map.

On the other hand, the mirror elliptic curve family under consideration is given by the equation $W=0$, which is \emph{not} the universal family over the moduli stack $\mathrm{SL}(2,\mathbb{Z})\backslash \mathcal{H}^{*}$ of complex structures of the mirror elliptic curve.
This elliptic curve family is essentially (up to reparametrization and base change, as shown in Section \ref{sectionGWpotentials}) the elliptic curve families of type $E_{n}$ reviewed in Section \ref{sectionmodularforms}. Note that the base change would also alter the modular group for which the parameter $\sigma$ in the elliptic curve family in \eqref{eqmirrorcurvefamily} is a Hauptmodul.  Since the parameter for the base of the family $W=0$ is a modular function for certain modular group, one would expect that the coefficients in the equation $W=0$, as functions on the modular curve, are related to modular forms. For example, in the $\mathbb{P}^{1}_{3,3,3}$ case, the equation $W=0$ defines the universal family of elliptic curves over the modular curve $\Gamma(3)\backslash\mathcal{H}^{*}$, and the parameters $\phi,\psi$ are modular forms for $\Gamma(3)$.  The big picture is illustrated in Figure \ref{chainofdualities}.\\

Now in order to see more clearly why it is the modular subgroup $\Gamma$ instead of the full modular group $\mathrm{SL}(2,\mathbb{Z})$ that enters the picture, the main point is as follows.  We have fixed the Seidel Lagrangian $\mathcal{L} \subset E/G$ to define the open Gromov-Witten potential.  The Lagrangian $\mathcal{L}$ lifts to $r$ copies of Lagrangians $L_1,\ldots,L_r$ in $E$, where $r = 3,4,6$ respectively.  Thus the A-side moduli under consideration is \emph{the K\"ahler structure together with the markings by these $r$ Lagrangians}.  By homological mirror symmetry, the corresponding B-side moduli for the mirror is the complex structure on $\check{E}$ together with the coherent sheaves mirror to $L_1,\ldots,L_r$.  In the next subsection, we show that these sheaves give rise to a cyclic subgroup of order $r$ of the group of $r$-torsion points on $\check{E}$.  Thus the moduli space is given by the modular curve $X_{\Gamma}= \Gamma\backslash\mathcal{H}^{*}$ instead of $\mathrm{SL}(2,\mathbb{Z})\backslash \mathcal{H}^{*}$.

\subsection{T-duality}
\label{sectionAbelJacobi}

It is a standard fact that the modular curve $\Gamma_{0}(r)\backslash \mathcal{H}^{*}$ is the (coarse) moduli space of pairs
$(E, H)$, where $E$ is an elliptic curve and $H<E_{r}$ is a cyclic subgroup of order $r$ of the group of $r$-torsion points on $E$.

For simplicity, we focus on $\mathbb{P}^{1}_{3,3,3}$, and the other two cases are similar.  The Seidel Lagrangian in $\mathbb{P}^{1}_{3,3,3}$ lifts to three Lagrangian cycles in the elliptic curve $E_{\rho}$ with its automorphism group generated by the cube root of unity $\rho=\exp (2\pi i/3)$. They are denoted as $\{L,\rho L,\rho^{2}L\}$, with
\begin{equation}
[L]=A+2B\,, \quad [\rho L]=-2A-B\,,\quad [\rho^{2}L]=A-B\,,
\end{equation}
where $A,B \in H_1(E_{\rho},\mathbb{Z})$ are the generators corresponding to the lattice points $1$ and $\rho$ which give rise to the elliptic curve $E_{\rho}$, respectively.

We will use T-duality to transform $\{L,\rho L,\rho^{2}L\}$ to coherent sheaves on the mirror elliptic curve $\check{E}_{\rho}$.  T-duality and homological mirror symmetry for elliptic curves was well-studied, see for instance \cite{PZ}, and we include it here for completeness of the paper.

To avoid dealing with multi-sections, we consider the double cover $\tilde{E}_{\rho}$ of the elliptic curve $E_{\rho}$
with its corresponding lattice generated by $2, \rho$. The Lagrangians
$L_{1}=L,L_{2}=\rho L,L_{3}=\rho^{2}L$ lifts to Lagrangians $\tilde{L}_{1},\tilde{L}_{2},\tilde{L}_{3}$ in the double cover.
Take the generators of $H_{1}(\tilde{E}_{\rho})$ to be $\tilde{A},B$ corresponding to the lattice points $2, \rho$.
Then we have 
\begin{equation}
[\tilde{L}_{1}]=\tilde{A}+4B, [\tilde{L}_{2}]=\tilde{A}+B, [\tilde{L}_{3}]=\tilde{A}-2B\,.
\end{equation}
The intersections are
\begin{equation}
\tilde{L}_{1}\cap \tilde{L}_{2}=-3, \tilde{L}_{2}\cap \tilde{L}_{3}=-3, \tilde{L}_{3}\cap\tilde{L}_{1}=6\,.
\end{equation}
Let $s = \tilde{L}_{1}$ and $f=\tilde{B}=-3B$.  We then have
\begin{equation}
\tilde{L}_{1}=s,\tilde{L}_{2}=s+f,\tilde{L}_{3}=s+2f\,.
\end{equation}

Consider the elliptic curve $C$ whose lattice is generated by $2+4\rho, -3\rho$. Now $s$ and $f$ can be regarded as a section and a fiber of a Lagrangian fibration on this elliptic curve.  
By T-duality, they are mirror to the following sheaves on the mirror curve $\check{C}$:
$\mathcal{O}_{1}=\mathcal{O}, \mathcal{O}_{2}=\mathcal{O}(D),\mathcal{O}_{3}=\mathcal{O}(2D)$
where $D$ is the divisor of degree $1$ corresponding to the fiber class $f$ (equipped with trivial flat connection).  

The action which takes a Lagrangian section $s$ to $s+f$ corresponds to tensoring $\mathcal{O}(D)$ in the mirror curve $\check{C}$.  
The relation $\rho^3 = 1$ says the mirror $\mathbb{Z}_{3}$ action permutes $\mathcal{O},\mathcal{O}(D),\mathcal{O}(2D)$ cyclically.  
It follows that the sheaves give rise to a cyclic subgroup of order $3$ of the group of $3$-torsion points on the variety $\mathrm{Pic}^{0}(\check{C})$, which is isomorphic to the mirror elliptic curve $\check{C}$ itself.

To conclude, for the mirror side, we should consider the moduli space of complex structures of an elliptic curve decorated with a cyclic subgroup of order three of the group of $3$-torsion points on the elliptic curve.  Thus the global moduli is given by $\Gamma_{0}(3)\backslash \mathcal{H}^{*}$, and the open Gromov-Witten potential should be globally defined over $\Gamma_{0}(3)\backslash \mathcal{H}^{*}$. From previous sections we see that it is actually a global object over $\Gamma(3)\backslash \mathcal{H}^{*}$.

\subsection{One more example}

We now give one more example for which the global moduli space of K\"ahler structures can be identified with a modular curve and the 
generating functions of Gromov-Witten invariants are modular forms.

The mirror manifold of $K_{\bP^2}$ is a non-compact Calabi-Yau 3-fold $X$ given by \cite{Hori:2000kt}
\begin{equation}
\{uv = 1 + z + w +
\alpha / zw\} \subseteq \mathbb{C}^{2}_{u,v}\times ({\mathbb{C}^{\times}})^{2}_{z,w}\,,
\end{equation} 
and is a conic fibration over the base $({\mathbb{C}^{\times}})^{2}_{z,w}$.
The flat coordinate, denoted by $t(\alpha)$, for the threefold $X$
can be expressed in terms of the flat coordinate $\tau(\alpha)$ for
the corresponding elliptic curve $\{1+z+w+\alpha/zw=0\} \subset (\C^\times)^2_{z,w}$ which is the discriminant locus of the conic fibration.  

The idea is the following. On one hand, $\alpha(\tau)$ is automatically a modular form as it is the Hauptmodul for the modular curve $\Gamma_{0}(3)\backslash\mathcal{H}^{*}$ which parametrizes the elliptic curve family above, see \cite{Alim:2013eja}.  Thus it is a
tautology that $\alpha(t(\tau))$ is a modular form.  On the other
hand, in the A-model on  $K_{\bP^2}$, we know that $\alpha(t)$ is a generating function of open
Gromov-Witten invariants \cite{CLL}.  Therefore, we know that the generating function of
open Gromov-Witten invariants of $K_{\bP^2}$ is a modular form defined over the complexified K\"ahler
moduli space, which under mirror symmetry is identified with the modular curve $\Gamma_{0}(3)\backslash\mathcal{H}^{*}$ parametrizing the mirror manifolds of $K_{\bP^2}$.

The details are given as follows. The SYZ mirror Calabi-Yau 3-fold $X$ for $K_{\mathbb{P}^{2}}$ is given by \cite{CLL}
\begin{equation}
w_{1}w_{2}=1+\delta(q)+z_{1}+z_{2}+{q\over z_{1}z_{2}}\,,
\end{equation}
with 
\begin{equation}\label{eqopeninvariantsKP2}
1+\delta(q)=\sum_{k=0}^{\infty} n_{k}q^{k}\,,
\end{equation} 
where
$q=q_{t}:=\exp 2\pi i t$, $t$ is the flat coordinate on the complexified K\"ahler
moduli space of $K_{\mathbb{P}^{2}}$.
Then the mirror curve is given by $1+\delta(q)+z_{1}+z_{2}+{q\over
z_{1}z_{2}}=0$. A scaling on the coordinates shows that this
curve is equivalent to
\begin{equation}\label{eqeqnforzSYZ}
1+z_{1}+z_{2}+{z\over z_{1}z_{2}}=0\,,\quad z={q_{t}\over
(1+\delta(q_{t}))^{3}}\,.
\end{equation}
Now consider $z$ as the complex structure modulus for the mirror curve.
It is a standard fact that this elliptic curve family is 3-isogenous to the $\tilde{E}_{6}$ curve family in Section \ref{sectiongeometricmoduli} and thus is parametrized by the modular curve $\Gamma_{0}(3)\backslash \mathcal{H}^{*}$.
Furthermore, one has
\begin{equation}\label{eqeqnforzmodular}
z(\tau) =-{\alpha(\tau)\over 27}=-{1\over 27} {({3 \eta(3\tau)^{3}\over
\eta(\tau)})^{3}  \over  ({3 \eta(3\tau)^{3}\over \eta(\tau)})^{3}
+({ \eta(\tau)^{3}\over \eta(3\tau)})^{3}   }\,.
\end{equation}
The relation between the modular variable $q_{\tau}:=\exp 2\pi i
\tau$ and the flat coordinate $t$ is given by \cite{Mohri:2000kf, Stienstra:2005wy, Zhou:2014thesis},
\begin{equation}
q_{\tau}=(-q_{t})\prod_{d\geq 1}(1-q_{t}^{d})^{3d^{2}n^{\mathrm{GV}}_{0,d}}\,, \quad 
q_{t}=(-q_{\tau}) \prod_{n\geq 1} (1-q_{\tau}^{n})^{9n\chi_{-3}(n)}\,.
\end{equation}
where $n^{\mathrm{GV}}_{0,d}=3,-6,27,-192,1695\cdots$ are the genus $0$ degree $d$
Gopakumar-Vafa invariants \cite{Gopakumar:1998ii, Gopakumar:1998jq, Katz:1999xq}, and $\chi_{-3}(n)$ is the non-trivial Dirichlet character mod $3$ (it takes the value $0,1,-1$ on an integer $3k,3k+1,3k+2$, respectively).
From the above formulas in \eqref{eqeqnforzSYZ}, \eqref{eqeqnforzmodular} for the same quantity $z$, one then has
\begin{equation}
1+\delta(q_{t})=(-27)^{1\over 3}q_{t}^{1\over 3}  \alpha(q_{\tau})^{-{1\over
3}}=(-27)^{1\over 3}q_{t}^{1\over 3}  \alpha(q_{\tau}(q_{t}))^{-{1\over 3}}\,.
\end{equation}
The first few constants $\{n_{k}\}_{k\geq 0}=\{1,-2,5,-32,286,-3038,35870\cdots \}$ predicted by using this formula and \eqref{eqopeninvariantsKP2} give exactly the open Gromov-Witten invariants computed by a direct counting as done in \cite{CLL}. That is, the generating function $1+\delta(q_{t})$, up to
multiplication by the factor $q_{t}^{1/ 3}$, is a modular
form in $q_{\tau}$.


\providecommand{\bysame}{\leavevmode\hbox to3em{\hrulefill}\thinspace}
\providecommand{\MR}{\relax\ifhmode\unskip\space\fi MR }
\providecommand{\MRhref}[2]{%
  \href{http://www.ams.org/mathscinet-getitem?mr=#1}{#2}
}
\providecommand{\href}[2]{#2}

\bigskip{}

\noindent{\small Department of Mathematics, Harvard University, One Oxford Street, Cambridge, MA 02138, USA}

\noindent{\small E-mail: \tt s.lau@math.harvard.edu}

\medskip{}
\noindent{\small Perimeter Institute for Theoretical Physics, 31 Caroline Street North, Waterloo, ON N2L 2Y5, Canada}

\noindent{\small E-mail: \tt jzhou@perimeterinstitute.ca}


\begin{thebibliography}{AGNT95}

\bibitem[ABK08]{Aganagic:2006wq}
M. Aganagic, V. Bouchard, and A. Klemm, \emph{{Topological Strings
  and (Almost) Modular Forms}}, Commun.Math.Phys. \textbf{277} (2008),
  771--819.

\bibitem[AGNT95]{Antoniadis:1995zn}
I. Antoniadis, E.~Gava, K.~S. Narain, and T.~R. Taylor, \emph{{N=2 type II
  heterotic duality and higher derivative F terms}}, Nucl.Phys. \textbf{B455}
  (1995), 109--130.

\bibitem[AJ10]{AJ}
M. Akaho and D. Joyce, \emph{Immersed {L}agrangian {F}loer theory}, J.
  Differential Geom. \textbf{86} (2010), no.~3, 381--500.

\bibitem[AS12]{Alim:2012ss}
M. Alim and E. Scheidegger, \emph{{Topological Strings on Elliptic
  Fibrations}},  ArXiv e-prints (2012).

\bibitem[ASYZ14]{Alim:2013eja}
M. Alim, E. Scheidegger, S.-T. Yau, and Jie Zhou, \emph{Special
  polynomial rings, quasi modular forms and duality of topological strings},
  Adv. Theor. Math. Phys. \textbf{18} (2014), no.~2, 401--467. \MR{3273318}

\bibitem[BB91]{Borwein:1991}
J.~M. Borwein and P.~B. Borwein, \emph{A cubic counterpart of {J}acobi's
  identity and the {AGM}}, Trans. Amer. Math. Soc. \textbf{323} (1991), no.~2,
  691--701. \MR{1010408 (91e:33012)}

\bibitem[BBG94]{Borwein:1994}
J.~M. Borwein, P.~B. Borwein, and F.~G. Garvan, \emph{Some cubic modular
  identities of {R}amanujan}, Trans. Amer. Math. Soc. \textbf{343} (1994),
  no.~1, 35--47. \MR{1243610 (94j:11019)}

\bibitem[BBG95]{Berndt:1995}
B.~C. Berndt, S.~Bhargava, and F.~G. Garvan, \emph{Ramanujan's theories
  of elliptic functions to alternative bases}, Trans. Amer. Math. Soc.
  \textbf{347} (1995), no.~11, 4163--4244. \MR{1311903 (97h:33034)}

\bibitem[BCOV93]{Bershadsky:1993ta}
M.~Bershadsky, S.~Cecotti, H.~Ooguri, and C.~Vafa, \emph{{Holomorphic anomalies
  in topological field theories}}, Nucl.Phys. \textbf{B405} (1993), 279--304.

\bibitem[BCOV94]{Bershadsky:1993cx}
\bysame, \emph{{Kodaira-Spencer theory of gravity and exact results for quantum
  string amplitudes}}, Commun.Math.Phys. \textbf{165} (1994), 311--428.

\bibitem[BKMS01]{Bannai:2001}
E. Bannai, M. Koike, A. Munemasa, and J. Sekiguchi, \emph{Some
  results on modular forms---subgroups of the modular group whose ring of
  modular forms is a polynomial ring}, Groups and combinatorics---in memory of
  {M}ichio {S}uzuki, Adv. Stud. Pure Math., vol.~32, Math. Soc. Japan, Tokyo,
  2001, pp.~245--254. \MR{1893493 (2003c:11032)}

\bibitem[Bri07]{Bridgeland}
T. Bridgeland, \emph{Stability conditions on triangulated categories}, Ann. of
  Math. (2) \textbf{166} (2007), no.~2, 317--345.

\bibitem[CHKL14]{CHKL}
C.-H. Cho, H. Hong, S.-H. Kim, and S.-C. Lau,
  \emph{Lagrangian {F}loer potential of orbifold spheres}, ArXiv e-prints
  (2014).

\bibitem[CHL13]{CHL}
C.-H. Cho, H. Hong, and S.-C. Lau, \emph{Localized mirror functor for Lagrangian immersions, and homological mirror symmetry for $\bP^1_{a,b,c}$}, ArXiv e-prints (2013).

\bibitem[CHL15]{CHLnc}
C.-H. Cho, H. Hong, and S.-C. Lau, \emph{Non-commutative mirror functors}, in preparation.

\bibitem[CKYZ99]{Chiang:1999tz}
T.-M. Chiang, A.~Klemm, S.-T. Yau, and E.~Zaslow, \emph{{Local mirror
  symmetry: Calculations and interpretations}}, Adv.Theor.Math.Phys. \textbf{3}
  (1999), 495--565.

\bibitem[CLL12]{CLL}
K. Chan, S.-C. Lau, and N.-C. Leung, \emph{S{YZ} mirror
  symmetry for toric {C}alabi-{Y}au manifolds}, J. Differential Geom.
  \textbf{90} (2012), no.~2, 177--250.
  
 \bibitem[Con96]{Connell:1996}
I. Connell, \emph{Elliptic curve handbook}, Preprint, 1996.


\bibitem[Dij95]{Dijkgraaf:1995}
R. Dijkgraaf, \emph{Mirror symmetry and elliptic curves}, The moduli space
  of curves ({T}exel {I}sland, 1994), Progr. Math., vol. 129, Birkh\"auser
  Boston, Boston, MA, 1995, pp.~149--163. \MR{1363055 (96m:14072)}

\bibitem[EO01]{Eskin:2001}
A. Eskin and A. Okounkov, \emph{Asymptotics of numbers of branched
  coverings of a torus and volumes of moduli spaces of holomorphic
  differentials}, Invent. Math. \textbf{145} (2001), no.~1, 59--103.
  \MR{1839286 (2002g:32018)}

\bibitem[ET13]{ET}
W. Ebeling and A. Takahashi, \emph{Mirror symmetry between orbifold
  curves and cusp singularities with group action}, Int. Math. Res. Not. IMRN
  (2013), no.~10, 2240--2270.

\bibitem[FOOO09]{FOOO}
K. Fukaya, Y.-G. Oh, H. Ohta, and K. Ono, \emph{Lagrangian
  intersection {F}loer theory: anomaly and obstruction. {P}arts {I} and {II}},
  AMS/IP Studies in Advanced Mathematics, vol.~46, American Mathematical
  Society, Providence, RI, 2009.

\bibitem[GV98a]{Gopakumar:1998ii}
R. Gopakumar and C. Vafa, \emph{{M theory and topological strings.
  1.}}, ArXiv e-prints (1998).

\bibitem[GV98b]{Gopakumar:1998jq}
\bysame, \emph{{M theory and topological strings. 2.}}, ArXiv e-prints (1998).

\bibitem[HV00]{Hori:2000kt}
K. Hori and C. Vafa, \emph{{Mirror symmetry}}, ArXiv e-prints (2000).

\bibitem[KKRS05]{Klemm:2004km}
A.~Klemm, M.~Kreuzer, E.~Riegler, and E.~Scheidegger, \emph{{Topological string
  amplitudes, complete intersection Calabi-Yau spaces and threshold
  corrections}}, JHEP \textbf{0505} (2005), 023.

\bibitem[KKV99]{Katz:1999xq}
S.~H. Katz, A. Klemm, and C. Vafa, \emph{{M theory, topological
  strings and spinning black holes}}, Adv.Theor.Math.Phys. \textbf{3} (1999),
  1445--1537.

\bibitem[KLRY96]{Klemm:1996}
A.~Klemm, B.-H. Lian, S.-S. Roan, and S.-T. Yau, \emph{A note on {ODE}s from
  mirror symmetry}, Functional analysis on the eve of the 21st century, {V}ol.\
  {II} ({N}ew {B}runswick, {NJ}, 1993), Progr. Math., vol. 132, Birkh\"auser
  Boston, Boston, MA, 1996, pp.~301--323. \MR{1389022 (97h:32032)}

\bibitem[KM08]{Klemm:2005pd}
A. Klemm and M. Marino, \emph{{Counting BPS states on the enriques
  Calabi-Yau}}, Commun.Math.Phys. \textbf{280} (2008), 27--76.

\bibitem[KMW12]{Klemm:2012sx}
A. Klemm, J. Manschot, and T. Wotschke, \emph{{Quantum geometry of
  elliptic Calabi-Yau manifolds}},  ArXiv e-prints (2012).


\bibitem[Kon95]{kontsevich94}
M. Kontsevich, \emph{Homological algebra of mirror symmetry}, Proceedings of
  the {I}nternational {C}ongress of {M}athematicians, {V}ol.\ 1, 2 ({Z}\"urich,
  1994) (Basel), Birkh\"auser, 1995, pp.~120--139.

\bibitem[KS11]{KrSh}
M. Krawitz and Y. Shen, \emph{{L}andau-{G}inzburg/{C}alabi-{Y}au
  correspondence of all genera for elliptic orbifold {$\mathbb{P}^1$}},
 ArXiv e-prints (2011).

\bibitem[KV95]{Kachru:1995wm}
S. Kachru and C. Vafa, \emph{{Exact results for N=2 compactifications
  of heterotic strings}}, Nucl.Phys. \textbf{B450} (1995), 69--89.

\bibitem[KZ95]{Kaneko:1995}
M. Kaneko and D. Zagier, \emph{A generalized {J}acobi theta function and
  quasimodular forms}, The moduli space of curves ({T}exel {I}sland, 1994),
  Progr. Math., vol. 129, Birkh\"auser Boston, Boston, MA, 1995, pp.~165--172.
  \MR{1363056 (96m:11030)}

\bibitem[Li12]{Li:2011mi}
S.~Li, \emph{Feynman graph integrals and almost modular forms},
  Commun.Num.Theor.Phys. \textbf{6} (2012), 129--157.

\bibitem[LLS13]{Li:2013}
C.~{Li}, S.~{Li}, and K.~{Saito}, \emph{{Primitive forms via polyvector
  fields}}, ArXiv e-prints (2013).

\bibitem[LMW97]{Lerche:1996ni}
W.~Lerche, P.~Mayr, and N.P. Warner, \emph{{Noncritical strings, Del Pezzo
  singularities and Seiberg-Witten curves}}, Nucl.Phys. \textbf{B499} (1997),
  125--148.

\bibitem[LY96]{Lian:1994zv}
B.-H. Lian and S.-T. Yau, \emph{{Arithmetic properties of mirror map and
  quantum coupling}}, Commun.Math.Phys. \textbf{176} (1996), 163--192.

\bibitem[Mai09]{Maier:2009}
R.~S. Maier, \emph{On rationally parametrized modular equations}, J.
  Ramanujan Math. Soc. \textbf{24} (2009), no.~1, 1--73. \MR{2514149
  (2010f:11060)}

\bibitem[Mai11]{Maier:2011}
\bysame, \emph{Nonlinear differential equations satisfied by certain classical
  modular forms}, Manuscripta Math. \textbf{134} (2011), no.~1-2, 1--42.
  \MR{2745252 (2012d:11095)}

\bibitem[MM99]{Marino:1998pg}
M. Marino and G.~W. Moore, \emph{{Counting higher genus curves in a
  Calabi-Yau manifold}}, Nucl.Phys. \textbf{B543} (1999), 592--614.

\bibitem[MOY01]{Mohri:2000kf}
K. Mohri, Y. Onjo, and S.-K. Yang, \emph{{Closed submonodromy problems,
  local mirror symmetry and branes on orbifolds}}, Rev.Math.Phys. \textbf{13}
  (2001), 675--715.

\bibitem[MR11]{Milanov:2011}
T.~{Milanov} and Y.~{Ruan}, \emph{{Gromov-Witten theory of elliptic orbifold
  $\bP^{1}$ and quasi-modular forms}}, ArXiv e-prints (2011).

\bibitem[MS12]{Milanov:2012}
T.~{Milanov} and Y.~{Shen}, \emph{{Global mirror symmetry for invertible simple
  elliptic singularities}}, ArXiv e-prints (2012).

\bibitem[MT08]{milanov-tseng08}
T.~E. Milanov and H.-H. Tseng, \emph{The spaces of {L}aurent
  polynomials, {G}romov-{W}itten theory of {$\Bbb P^1$}-orbifolds, and
  integrable hierarchies}, J. Reine Angew. Math. \textbf{622} (2008), 189--235.
  \MR{2433616 (2010e:14053)}

\bibitem[OP06]{Okunkov:2002}
A.~Okounkov and R.~Pandharipande, \emph{Gromov-{W}itten theory, {H}urwitz
  theory, and completed cycles}, Ann. of Math. (2) \textbf{163} (2006), no.~2,
  517--560. \MR{2199225 (2007b:14123)}

\bibitem[Orl09]{Orlov}
D. Orlov, \emph{Derived categories of coherent sheaves and triangulated
  categories of singularities}, Algebra, arithmetic, and geometry: in honor of
  {Y}u. {I}. {M}anin. {V}ol. {II}, Progr. Math., vol. 270, Birkh\"auser Boston
  Inc., Boston, MA, 2009, pp.~503--531.

\bibitem[PT14]{Pandharipande:2014}
R.~{Pandharipande} and R.~P. {Thomas}, \emph{{The Katz-Klemm-Vafa conjecture
  for K3 surfaces}}, ArXiv e-prints (2014).

\bibitem[PZ98]{PZ}
A. Polishchuk and E. Zaslow, \emph{Categorical mirror symmetry: the
  elliptic curve}, Adv. Theor. Math. Phys. \textbf{2} (1998), no.~2, 443--470.

\bibitem[Ran77]{Rankin:1977ab}
R.~A. Rankin, \emph{Modular forms and functions}, Cambridge University
  Press, Cambridge, 1977.

\bibitem[Ros10]{Rossi}
P. Rossi, \emph{Gromov-{W}itten theory of orbicurves, the space of
  tri-polynomials and symplectic field theory of {S}eifert fibrations}, Math.
  Ann. \textbf{348} (2010), no.~2, 265--287.

\bibitem[Sai74]{Saito:1974}
K. Saito, \emph{Einfach-elliptische singularit\:aten}, Invent. Math.
  \textbf{23} (1974), 289--325.

\bibitem[Seb02]{Sebbar:2002}
A. Sebbar, \emph{Modular subgroups, forms, curves and surfaces}, Canad.
  Math. Bull. \textbf{45} (2002), no.~2, 294--308. \MR{1904094 (2003d:20064)}

\bibitem[Sei11]{Seidel:g=2}
P. Seidel, \emph{Homological mirror symmetry for the genus two curve}, J.
  Algebraic Geom. \textbf{20} (2011), no.~4, 727--769.

\bibitem[ST11]{Satake:2011}
I. Satake and A. Takahashi, \emph{Gromov-{W}itten invariants for mirror
  orbifolds of simple elliptic singularities}, Ann. Inst. Fourier (Grenoble)
  \textbf{61} (2011), no.~7, 2885--2907. \MR{3112511}

\bibitem[Sti06]{Stienstra:2005wy}
J. Stienstra, \emph{Mahler measure variations, {E}isenstein series and
  instanton expansions}, Mirror symmetry. {V}, AMS/IP Stud. Adv. Math.,
  vol.~38, Amer. Math. Soc., Providence, RI, 2006, pp.~139--150. \MR{2282958
  (2008d:11095)}

\bibitem[SZ14]{Shen:2014}
Y. {Shen} and J. {Zhou}, \emph{{Ramanujan Identities and Quasi-Modularity
  in Gromov-Witten Theory}}, ArXiv e-prints (2014).

\bibitem[Tak10]{Ta}
A. Takahashi, \emph{Weighted projective lines associated to regular
  systems of weights of dual type}, New developments in algebraic geometry,
  integrable systems and mirror symmetry ({RIMS}, {K}yoto, 2008), Adv. Stud.
  Pure Math., vol.~59, Math. Soc. Japan, Tokyo, 2010, pp.~371--388.

\bibitem[Zho14]{Zhou:2014thesis}
J. Zhou, \emph{{Arithmetic Properties of Moduli Spaces and Topological String
  Partition Functions of Some Calabi-Yau Threefolds}}, Harvard Ph. D. Thesis
  (2014).

\end{thebibliography}
\end{document}